 \documentclass[11pt]{article}
 \usepackage{hyperref} %for contents
  \usepackage{epsfig}
\usepackage{amssymb,amsmath,amsthm,amscd}
\usepackage{latexsym}
\usepackage{fixltx2e}
\usepackage{mathrsfs}
\usepackage{verbatim}
\usepackage{enumerate}
\usepackage{xcolor}

\usepackage[text={175mm,215mm},centering]{geometry}

\allowdisplaybreaks[4]

\newcommand{\be}{\begin{equation}}
\newcommand{\ee}{\end{equation}}
\newcommand{\bes}{\begin{equation}\begin{aligned}}
\newcommand{\ees}{\end{aligned}\end{equation}}
\newcommand{\ben}{\begin{equation}\nonumber\begin{aligned}}

\renewcommand{\leq}{\leqslant}
\renewcommand{\geq}{\geqslant}  %package{amssymb}

 \newtheorem{theorem}{Theorem}[section]
\newtheorem{lemma}[theorem]{Lemma}
\newtheorem{proposition}[theorem]{Proposition}

\newtheorem{definition}[theorem]{Definition}

\numberwithin{equation}{section}

\begin{document}

\baselineskip=1.3\baselineskip

\pagestyle{plain}

\title{ {Limiting Behavior of Non-Autonomous Stochastic  Reversible Selkov Lattice Systems Driven by Locally Lipschitz L\'{e}vy Noises}}

\author{
Guofu Li\footnote{School of Mathematical Sciences, Guizhou Normal University, Guiyang 550025, China, Email: liguoful@163.com.}\;,
Jianxin Wu\footnote{School of Mathematical Sciences, Guizhou Normal University, Guiyang 550025, China, Email: jianxinwu2001@163.com.}\;
Yunshun Wu\footnote{School of Mathematical Sciences, Guizhou Normal University, Guiyang 550025, China, Email: wuyunshun1979@163.com.}\;
}
\date{}
\maketitle

\medskip

\begin{abstract}
This work investigates the long-term distributional behavior of the reversible Selkov lattice systems defined on the set $\mathbb{Z}$ and driven by locally Lipschitz \emph{L\'{e}vy noises}, which possess two pairs of oppositely signed nonlinear terms and whose nonlinear couplings can grow polynomially with any order $p \geq 1$. Firstly, based on the global-in-time well-posedness in $L^{2}(\Omega, \ell^2 \times \ell^2)$, we define a \emph{continuous} non-autonomous dynamical system (NDS) on the metric space $(\mathcal{P}_{2}(\ell^2 \times \ell^2), d_{\mathcal{P}(\ell^2 \times \ell^2)})$, where $d_{\mathcal{P}(\ell^2 \times \ell^2)}$ is the dual-Lipschitz distance on $\mathcal{P}(\ell^2 \times \ell^2)$, the space of probability measures on $\ell^2 \times \ell^2$. Specifically, we establish that this non-autonomous dynamical system admits a  unique pullback measure attractor, characterized via measure-valued complete solutions and orbits in the sense of Wang (DOI.org/10.1016/j.jde.2012.05.015). Moreover, when the deterministic external forcing terms are periodic in time, we demonstrate that the pullback measure attractors are also periodic. We also study the upper semicontinuity of pullback measure attractors as $(\epsilon_1, \epsilon_2, \gamma_1, \gamma_2) \rightarrow (0, 0, 0, 0)$. The main difficulty in proving the pullback asymptotic compactness of the NDS in $(\mathcal{P}_{2}(\ell^2 \times \ell^2), d_{\mathcal{P}(\ell^2 \times \ell^2)})$ is caused by the lack of compactness in infinite-dimensional lattice systems, which is overcome by using uniform tail-ends estimates. And the inherent structure of the Selkov system precludes the possibility of any unidirectional dissipative influence arising from the interaction between the two coupled equations, thereby obstructing the emergence of a dominant energy-dissipation mechanism along a single directional pathway. Compared to the classical case of cubic nonlinearity, our results are significantly more general in scope and applicability.
\end{abstract}

{\bf Key words:} Selkov systems, upper semicontinuity, L\'{e}vy noises, Infinite-dimensional coupled lattice systems.

{\bf MSC 2020:} 37L55, 35B41, 37L60, 37B40, 60H10.

\vskip10mm

\section{Introduction}
\subsection{ Infinite-dimensional reversible Selkov model
  driven by   \emph{L\'{e}vy
noises}}

The Selkov equation is a significant and widely-used mathematical model for describing autocatalytic biochemical processes, such as glycolysis, which can be categorized into irreversible and reversible systems. The irreversible Selkov equation was initially introduced by Selkov in 1968 to model two irreversible chemical reactions (see \cite{37}). In contrast, the reversible Selkov model is primarily designed to represent the cubic autocatalytic, isothermal, and continuously supplied nature of two reversible biochemical reactions:
\begin{align}
 \hat{A} \rightleftharpoons \hat{S}, \quad \hat{S} + 2\hat{P} \rightleftharpoons 3\hat{P}, \quad \hat{P} \rightleftharpoons \hat{B},
\end{align} where $\hat{A}$ and $\hat{B}$ are controllable bath concentrations, and the substrate $\hat{S}$ and the product $\hat{P}$ form the freely responding internal part of the system. This model represents a strongly product-activated mechanism, which, within a suitable range of $\hat{A}$ and $\hat{B}$ concentrations, can give rise to oscillatory behavior. The reversible Selkov model is also known as the two-component Gray-Scott equations (see \cite{Art,Gray,Gray1}), which is part of the Brussels school, led by Ilya Prigogine, the famous physical chemist and Nobel Prize laureate (1977).

The motivation of this paper is to investigate the long-time distribution behaviors by means of pullback measure attractors of  reversible Selkov lattice systems defined on the integer set $\mathbb{Z}$ and driven by locally Lipschitz \emph{L\'{e}vy noises}:
\begin{align}\label{a1}
\left\{
\begin{aligned}
\mathrm{d}u_i(t)&=(d_1(u_{i+1}(t)-2u_{i}+u_{i-1}(t))-a_1u_i(t)+b_1u^{2p}_i(t)v_i(t)-b_2u^{2p+1}_i(t)+f_{1i}(t))\mathrm{d}t\\
&\quad+\epsilon_1\sum\limits_{k=1}^\infty(h_{k,i}(t)+\delta_{k,i}\tilde{\sigma_k}(t,u_i(t)))\mathrm{d}W_k(t)\\
&\quad+\epsilon_2\sum\limits_{k=1}^\infty\int_{|y_k|<1}(\kappa_{k,i}(t)+q_{k,i}(t,u_i(t-),y_k))\tilde{L}_k(\mathrm{d}t,\mathrm{d}y_k)\\
\mathrm{d}v_i(t)&=(d_2(v_{i+1}(t)-2v_i(t)+v_{i-1}(t))-a_2v_i(t)-b_1u^{2p}_i(t)v_i(t)+b_2u^{2p+1}_i(t)+f_{2i}(t))\mathrm{d}t\\
&\quad+\gamma_1\sum\limits_{k=1}^\infty(h_{k,i}(t)+\delta_{k,i}\tilde{\sigma_k}(t,v_i(t)))\mathrm{d}W_k(t)\\
&\quad+\gamma_2\sum\limits_{k=1}^\infty\int_{|y_k|<1}(\kappa_{k,i}(t)+q_{k,i}(t,v_i(t-),y_k))\tilde{L}_k(\mathrm{d}t,\mathrm{d}y_k)\\
\end{aligned}
\right.
\end{align}with initial conditions
\begin{align}\label{a2}
u_i(\tau)=u_{0,i},\ \ \ \ v_i(\tau)=v_{0,i},
\end{align}where $\tau\in\mathbb{R}$, $i\in\mathbb{Z}$, $
u=(u_i)_{i\in\mathbb{Z}}$, $v=(v_i)_{i\in\mathbb{Z}}\in\ell^2$, $0<\epsilon_i, \gamma_i\leq1$, $i=1,2$, $p\geq1$. $d_1$, $d_2$, $a_1$, $a_2$,
$b_1$, $b_2$
are positive constants, $f_1(t)=(f_{1i}(t))_{i\in\mathbb{Z}}$ and $f_2(t)=(f_{2i}(t))_{i\in\mathbb{Z}}$
$\in\ell^2$ are time dependent random sequences.
$\delta(t)=(\delta_{k,i}(t))_{k\in\mathbb{N},i\in\mathbb{Z}}$
is a  positive continuous functions. %sequence of real numbers satisfying
%$\|\delta\|^2:=\sum_{k\in\mathbb{N},i\in\mathbb{Z}}|\delta_{k,i}|^2<\infty$.
 $h(t)=(h_{k,i}(t))_{k\in\mathbb{N},i\in
\mathbb{Z}}$ and $\kappa(t)=(\kappa_{k,i}(t))_{k\in\mathbb{N},i\in\mathbb{Z}}$ are $\ell^2$-valued progressively measurable processes. $\sigma_k=\delta_{k,i}\tilde{\sigma_k}(t,\cdot)$ and $q_k=q_{k,i}(t,\cdot,y_k)$ are a sequence locally Lipschitz continuous functions.  \((W_k)_{k \in \mathbb{N}}\) denote a family of mutually independent, two-sided standard real-valued Wiener processes defined on  \((\Omega, \mathcal{F}, \{\mathcal{F}_t\}_{t \in \mathbb{R}}, \mathbb{P})\). The compensated Poisson random measure $\tilde{L}_k$ arising from a sequence of independent two-sides real-valued L\'{e}vy process $\{Z_k\}_{k\in \mathbb{Z}}$.

\textbf{L\'{e}vy noises}: The sequence $\{ W_k(t) \}$ represents independent two-sided real-valued Wiener processes defined on this probability space. We denote by $\{ Z_k(t) \}$ a sequence of independent two-sided real-valued L\'evy processes defined on the same probability space. A two-sided Poisson counting random measure is defined as
\[
L_k(t, A) =
\begin{cases}
\# \{ 0 \le s \le t : \Delta Z_k(s) \in A \}, & \text{if } t \ge 0, \\[6pt]
\# \{ t \le s \le 0 : \Delta Z_k(s) \in A \}, & \text{if } t < 0,
\end{cases}
\]
where $\#$ denotes the cardinality of a set, $A \in \mathcal{B}(\mathbb{R} \setminus \{0\})$, and $\Delta Z_k(s) = Z_k(s) - Z_k(s-).$ Let $\nu_k(dy_k)$ be a $\sigma$-finite L\'evy measure associated with a Poisson random measure $L_k(dt, dy_k)$, satisfying

\[
\nu_k(dy_k)\, dt = \mathbb{E}\bigl[ L_k(dt, dy_k) \bigr]
\quad \text{and} \quad
\sum_{k \in \mathbb{N}} \int_{\mathbb{R} \setminus \{0\}} \bigl( |y_k|^2 \wedge 1 \bigr) \, \nu_k(dy_k) < \infty.
\]
This implies that there exists a positive constant $M_{\text{jump}}$ such that
\[
\sum_{k \in \mathbb{N}} \int_{\mathbb{R} \setminus \{0\}} \bigl( |y_k|^2 \wedge 1 \bigr) \, \nu_k(dy_k) \le M_{\text{jump}}.
\]
The two-sided compensated Poisson random measure is denoted by
\[
\widetilde{L}_k(dt, dy_k) \triangleq L_k(dt, dy_k) - \nu_k(dy_k) dt.
\]
Throughout this paper, we assume that $\{ W_k \}_{k \in \mathbb{N}}$ and $\{ Z_k(t) \}$ are mutually independent. This implies that $\{ W_k \}_{k \in \mathbb{N}}$ and $\{ L_k(dt, dy_k) \}_{k \in \mathbb{N}}$ are mutually independent, the pair $((W_k)_{k\in\mathbb{N}}, (L_k)_{k\in\mathbb{N}})$ is called a family of \emph{two-sided} \emph{L\'{e}vy} noises, see \cite{Ly} and Figure $1$.

\vspace{-8pt}
\begin{center}
    \includegraphics[width=0.9\textwidth, height=5cm]{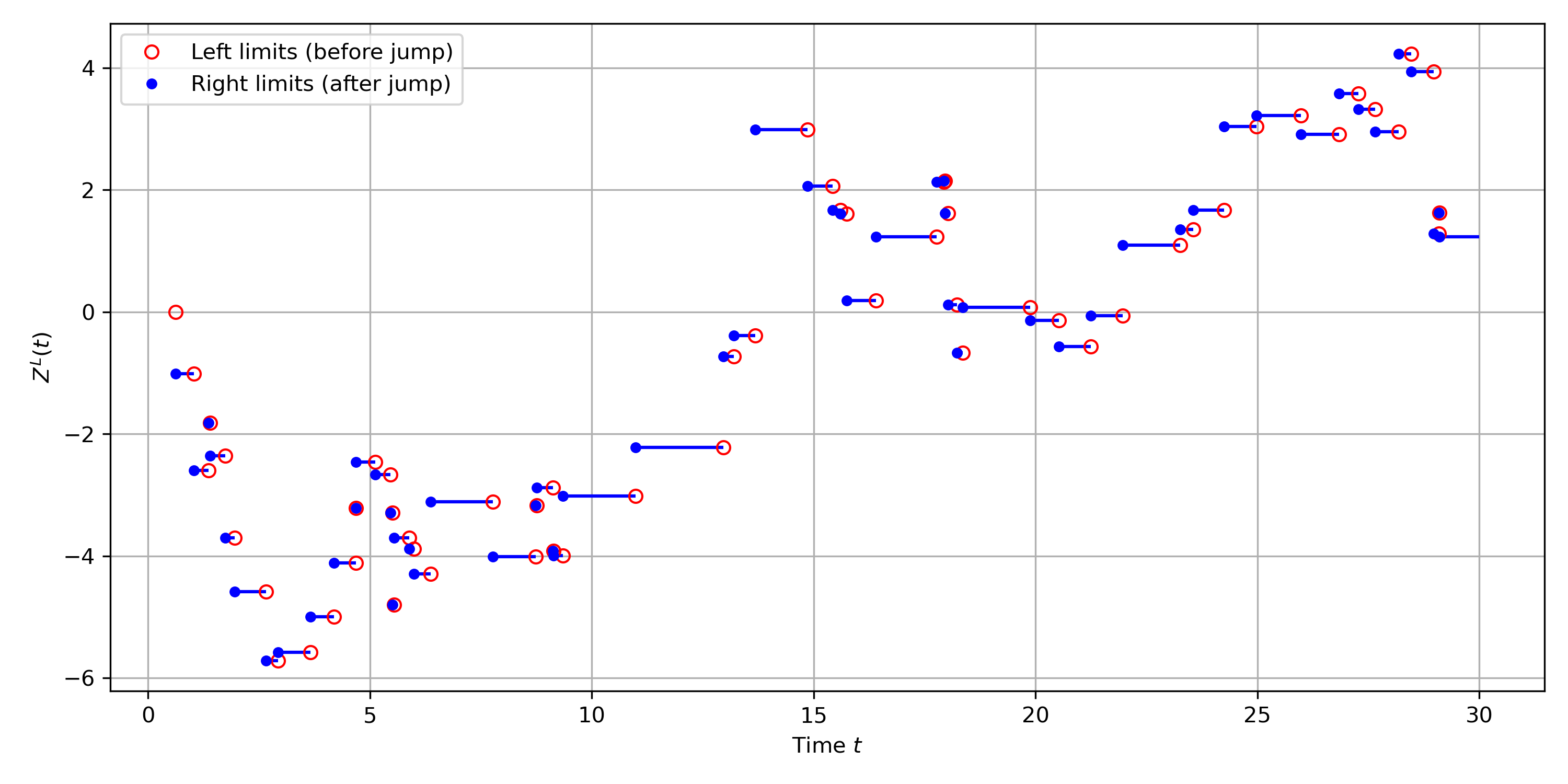}\\
    \small\textbf{Figure 1:} Sample paths of L\'{e}vy processes (\textbf{c\`{a}dl\`{a}g process})
    \label{fig:cadlag_path}
\end{center}
\vspace{-8pt}
\subsection{Literature survey and Motivations}
For the past few years, the widespread application of \textbf{lattice dynamical systems (LDSs)}, such as in pattern formation, neural pulse propagation, and circuits, has drawn considerable attention from many researchers, for more detailed information, please refer to \cite{1, 2,JJJ2}. The solutions of deterministic lattice systems have been thoroughly investigated in the literature, traveling wave solutions \cite{6, 7}, chaotic solutions \cite{8, 9}, and global attractors \cite{10, 11}.  For more details and results on the dynamics of the dynamics of stochastic lattice dynamical systems, see \cite{5,YYY00,YYY01,YYY02} for deterministic lattice models, and \cite{12,QAZ12,KKK00,KKK01,KKK02} for stochastic case.
When $\epsilon_2=0$, $\gamma_2=0$, the reversible Selkov lattice systems driven by white noises have been discussed by Wang and Zhang, see \cite{40,Zhang7}. Stochastic reversible Selkov lattice system driven by  L\'{e}vy noises \eqref{a1} reduces to the deterministic lattice system ($\epsilon_1=\epsilon_2=\gamma_1=\gamma_2=0$).
The deterministic reversible Selkov equations have been extensively studied with respect to the existence, regularity, and stability of global attractors (see \cite{24,25}). Numerous researchers have also explored the existence and robustness of random attractors in stochastic reversible Gray-Scott systems (see \cite{26,27,28,29}). For stochastic lattice reversible Selkov equations, the existence of random attractors was discussed by Li \cite{LSQAN}. In the latest study, Wang et al in \cite{num987} studied upper semi-continuity of numerical attractors for reversible Selkov lattice systems, for the relevant theories on numerical attractors for lattice systems, please refer to Li, Han, Kloeden Yang, \cite{LLNN987,LLNN988,LLNN990,LLNN989}.

\vspace{-8pt}
\begin{center}
    \includegraphics[width=0.9\textwidth, height=5cm]{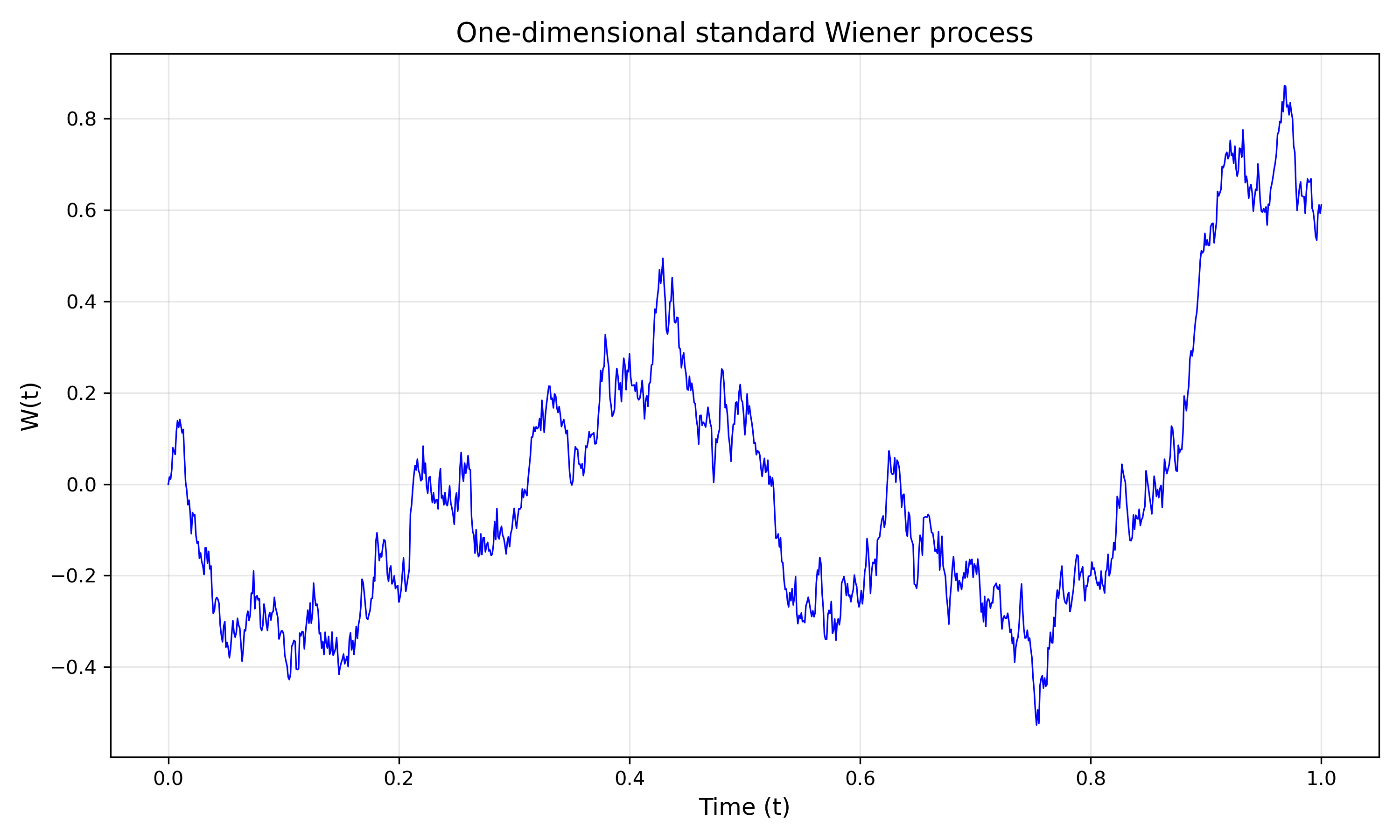}\\
    \small\textbf{Figure 2:} The sample path of \textbf{Brownian motion}
    \label{fig:levy_path}
\end{center}
\vspace{-8pt}

The lattice models mentioned above are driven by continuous stochastic processes (Brownian motion).  Nevertheless, many real-valued systems may be affected by sudden disturbances in the environment, leading to discontinuous sample paths.  Therefore, L\'{e}vy noise has to be introduced into the systems to describe random phenomena with jumps. To the best of our knowledge, there are not many literature results on measure attractors of Selkov lattice systems, and the current results are all devoted to the case where the type of noises  are driven by white noise rather than \emph{L\'{e}vy noises}. Our main interest in this article is to discuss the existence,
uniqueness, periodicity and upper semicontinuity of \emph{pullback measure  attractors} in  $(\mathcal{P}_{2}(\ell^2\times\ell^2),
d_{\mathcal{P}(\ell^2\times\ell^2)})$ for Selkov lattice systems \eqref{a1}
defined on $Z$ and driven by \emph{L\'{e}vy noise}. For the study on lattice models driven by L\'{e}vy noise,
the readers can refer to literature by Tom\'{a}s et al. \cite{levy1,levy2,OO546,YYUUU1} and the references therein for more details.

To study path-wise random attractors, one must first transform the stochastic evolutionary PDEs into pathwise evolutionary systems. When the noise coefficients are either additive or linearly multiplicative, such a transformation is possible through the well-known Ornstein-Uhlenbeck or Cole-Hopf transformations, see \cite{42,44,45,46,47,50}. However, these methods do not apply when the noise coefficients are nonlinear in the unknown function. In light of this, an alternative approach for analyzing the long-term dynamics of solutions to stochastic evolutionary PDEs with nonlinear noise coefficients is to adopt the concept of mean random attractors, as introduced by Kloeden, Lorenz, and Wang \cite{Kloeden, wangweak1, wangweak2, wangweak3}. Another promising approach to investigate the asymptotic behavior of solutions to nonlinear noise-driven PDEs involves the notion of measure attractors. Initially proposed by Schmalfu\ss \cite{Schmalfuss1, Schmalfuss2} to explore the long-term behavior of solutions to the \emph{autonomous} Navier-Stokes equations with nonlinear noise, this concept has recently been extended by Li and Wang \cite{LID1} (also see \cite{LID2, Shi}) to study the long-term behavior of solutions to \emph{nonautonomous} reaction-diffusion equations under nonlinear noise. This framework provides a more general and adaptable tool for analyzing complex stochastic Selkov lattice systems where traditional transformation methods are not applicable. The main focus of this research consists of three core aims, the details are as follows:
\begin{itemize}
    \item To establish with rigorous  theory of pullback measure attractors that the lattice systems defined by $\eqref{a1}$ not only possess unique solutions in \cite{LGF} but also exhibit the existence of pullback measure attractors;
    \item To provide a detailed demonstration for uniform moment estimates and $\mathcal{D}-$pullback asymptotic compact that are used to establish the existence and convergence of pullback measure attractors;
    \item To prove the upper semicontinuity of $\mathcal{D}$-pullback measure attractors for the non-autonomous stochastic lattice systems as the noise intensity  $(\epsilon_1,\epsilon_2,\gamma_1,\gamma_2)\rightarrow(0,0,0,0)$.
\end{itemize}

\subsection{The methods and techniques of the study}
\begin{itemize}
\item We remark that a key step for establishing the existence, uniqueness and periodicity of measure
pullback attractors of $S$ in $(\mathcal{P}_{2}(\ell^2\times\ell^2),
d_{\mathcal{P}(\ell^2\times\ell^2)})$ is to prove the
\emph{pullback asymptotic compactness} of $S$ (NDS) in $(\mathcal{P}_{2}(\ell^2\times\ell^2),
d_{\mathcal{P}(\ell^2\times\ell^2)})$, that is,  we shall establish the \emph{pullback asymptotic tightness} of probability distributions of solution operators of \eqref{b8} on    $(\mathcal{P}_{2}(\ell^2\times\ell^2),
d_{\mathcal{P}(\ell^2\times\ell^2)})$. In fact, the difficulty caused by the lack of compactness in infinite-dimensional lattice systems. Note that this type of obstacle is similar to that of showing the tightness of distribution of solutions to stochastic PDEs defined on unbounded domains, see \cite{QAZ12,wangweak2,QAZ14}.
For stochastic
evolutionary PDEs
defined on
bounded
 domains (see \cite{LID1,LID2,Schmalfuss1,Schmalfuss2}), one can use a compactness argument based on compact Sobolev embeddings to prove such asymptotic tightness. Since most problem is defined on $\mathbb{R}^d$, this compactness argument fails to establish such asymptotic tightness.  In fact, this idea was first developed by Wang \cite{Wangphysd} in proving the existence of global attractors of deterministic reaction-diffusion equations on $\mathbb{R}^d$, and then widely used to solve other problems with lack of compactness in random/stochastic cases such as the existence of pathwise random attractors and invariant measures of stochastic evolutionary PDEs and lattice systems on unbounded domains, see \cite{16,
wangweak2}.
Utilizing the uniform asymptotic tails-ends-smallness of the solutions in $L^2(\Omega,\ell^2\times\ell^2 )$,  we are able to establish the pullback asymptotic compactness
and the existence of a unique measure pullback attractor
for the $S$ in $(\mathcal{P}_{2}(\ell^2\times\ell^2),
d_{\mathcal{P}(\ell^2\times\ell^2)})$, see Theorem \ref{QQQ1}.

    \item For the analysis of Selkov lattice systems, we mathematically allow that the coupled drift term grows
    with an arbitrary index $p \geq 1$. This is different from the cubic nonlinearity case
    \cite{LSQAN}. In our case, we need to handle two pairs of oppositely signed drift terms
    $\pm F(u,v)$ and $\pm G(u)$ in the $u$-equation and the $v$-equation. This obstacle
    is overcame by carefully rearranging and analyzing these coupled nonlinear terms,
    see Lemma \ref{yingli41} and Lemma \ref{z42}.
   Note a single reaction-diffusion equation, the dissipativity is established by either imposing or checking the asymptotic sign condition on the nonlinear function \(f(u)\) representing the reaction rate, i.e.,
\[
\limsup_{|s|\to\infty}f(s)\,s\leq 0.
\]
But for most reaction-diffusion systems consisting of two or more equations arising from chemical and biochemical kinetics, the corresponding asymptotic sign condition in vector version is usually or inherently not satisfied.

   \item It is worth noting that this paper focuses on the pullback measure attractors of coupled lattice equations, which differs from the existing research on single-equation lattice systems in \cite{LGF21,QAZ12,mnbv2} and overcomes several essential difficulties. Compared to Brownian motion, a L\'{e}vy process is more general and may include jumps, with potentially discontinuous paths, whereas Brownian motion is merely the continuous special case of a L\'{e}vy process. In particular, the application of L\'{e}vy-type It\^{o}' formulas requires additional technical care, especially in establishing  uniform moment estimates  and proving the  upper semicontinuity of pullback measure attractors, where the presence of jumps introduces significant extra challenges.

\end{itemize}

\subsection{Main results}
Here are three results  results of this article, which establish the existence, uniqueness and  periodicity of the pullback measure attractors for the system \eqref{a1}, as well as their upper semicontinuity.
\begin{theorem}
Suppose \eqref{b6}-\eqref{bc0}, and \eqref{xs} hold. Then for every $0 < \epsilon_i,\gamma_i \leq 1$, $i=1,2$, the system $S$ associated with \eqref{b8}-\eqref{bc1} has a unique $\mathcal{D}$-pullback measure attractor $\mathcal{A} = \{\mathcal{A}(\tau) : \tau \in \mathbb{R}\} \in \mathcal{P}_2(\ell^2\times\ell^2)$, which is given by, for each $\tau \in \mathbb{R}$,
\begin{align}
\mathcal{A}(\tau) = \omega(K, \tau) = \{\psi(0, \tau) : \psi \text{ is a } \mathcal{D}\text{-complete orbit of } S\} \nonumber \\
= \{\xi(\tau) : \xi \text{ is a } \mathcal{D}\text{-complete solution of } S\},
\end{align}
where $K = \{K(\tau) : \tau \in \mathbb{R}\}$ is the $\mathcal{D}$-pullback absorbing set of $S$ as given by Lemma \ref{h17}.
\end{theorem}
\begin{theorem}
Suppose \eqref{b6}-\eqref{bc0}, \eqref{xs} and \eqref{xs1} hold. then for every $0 < \epsilon_i,\gamma_i \leq 1$, $i=1,2$, $S$ associated with \eqref{b8}-\eqref{bc1} has a unique $\chi$-periodic $\mathcal{D}$-pullback measure attractor $\mathcal{A}$ in $\mathcal{P}_2(\ell^2\times\ell^2)$.
\end{theorem}
\begin{theorem}
Suppose \eqref{b6}-\eqref{b7}, \eqref{bc0}, and \eqref{xs} hold. Then for $\tau \in \mathbb{R}$,
\begin{align}\label{h23}
\lim_{\lambda \to \lambda_0} d_{\mathcal{P}_2(\ell^2\times\ell^2)}(\mathcal{A}_{\lambda}(\tau), \mathcal{A}_{\lambda_0}(\tau)) = 0.
\end{align}
where $\lambda=(\epsilon_1,\epsilon_2,\gamma_1,\gamma_2)$, $\lambda_0=(0,0,0,0)$, for more details, we refer the reader to the finally section.
\end{theorem}
\vspace{-8pt}
\begin{center}
    \includegraphics[width=0.9\textwidth, height=5cm]{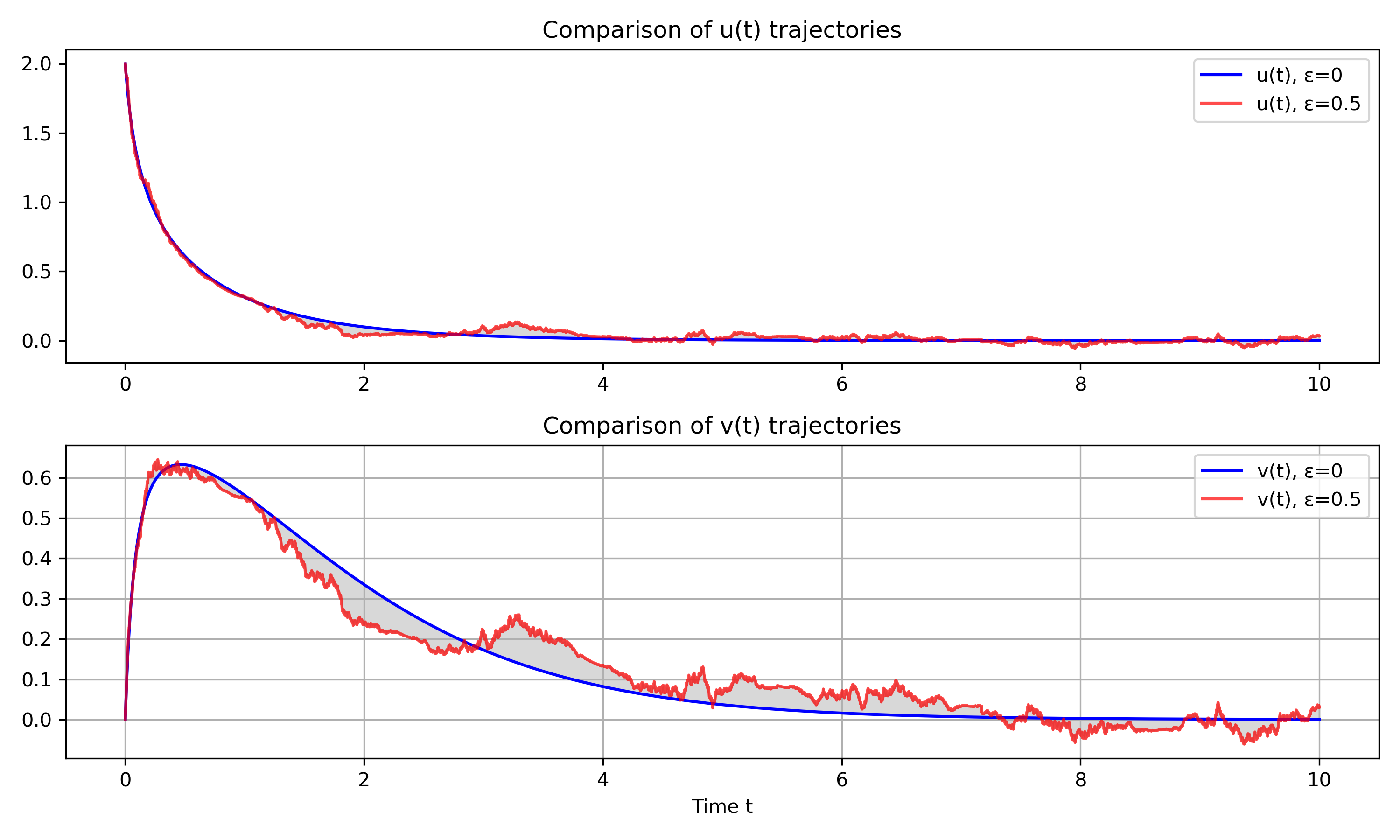}\\
    \small\textbf{Figure 3:} \textbf{$\lambda\rightarrow\lambda_0$}
    \label{upper semi-continuity}
\end{center}
\vspace{-8pt}

\subsection{Outline of paper}
The organization of this work is as follows. In Section 2, we explore the existence and uniqueness of solutions for the non-autonomous stochastic lattice system described by equations \eqref{b8} to \eqref{bc1}. In Section 3, we reconsider the key results concerning the existence, uniqueness, and periodicity of pullback measure attractors for non-autonomous dynamical systems, particularly those defined on the probability measure space within Banach spaces. Section 4 is dedicated to deriving uniform moment estimates for solutions as time \( t \to \infty \), which play a pivotal role in establishing the existence of absorbing sets as well as the pullback asymptotic compactness of the non-autonomous dynamical systems. These results are framed in the context of the Markov semigroup generated by the system \eqref{b8} to \eqref{bc1}. In the final three sections, we rigorously prove the existence, uniqueness, and periodicity of the pullback measure attractors for the system \eqref{b8} to \eqref{bc1}, and also investigate the behavior of these attractors in the limit as the parameters $(\epsilon_1,\epsilon_2,\gamma_1,\gamma_2)\rightarrow(0,0,0,0)$. In Section 7, we give a numerical simulation for a one-dimensional stochastic ODE \eqref{v0s}

\section{Existence and uniqueness of solutions}
In this section, we establish the existence and uniqueness of solutions to the system \eqref{a1}-\eqref{a2}. Some useful preliminaries and assumptions are provided, which will be frequently used in the subsequent discussions. Let $ \ell^r := \{ u = (u_i)_{i \in \mathbb{Z}} : \sum_{i \in \mathbb{Z}} |u_i|^r < +\infty \} $ for $ r \geq 1 $, with the norm of $ \ell^r $ denoted by $ \|\cdot\|_r $, while the norm and inner product of $ \ell^2 $ are denoted by $ \|\cdot\| $ and $ \langle \cdot, \cdot \rangle $, respectively.

Define the operators $A$, $B: \ell^2 \to \ell^2$ by
\begin{align}\label{b1}
(Au)_i &= -u_{i-1} + 2u_i - u_{i+1}, \quad \text{and} \quad (Bu)_i = u_{i+1} - u_i, \quad \text{for all } i \in \mathbb{Z}.
\end{align}Then, for all $u, v \in \ell^2$, the following relations hold:
\begin{align}\label{b2}
\langle Au, v \rangle &= \langle Bu, Bv \rangle, \quad \text{and} \quad \langle Au, u \rangle = \|Bu\|^2 \geq 0.
\end{align}Furthermore, for each $k \in \mathbb{N}$, define the operators $\sigma_k$, $q_k$, $F$, and $G$ as follows:
\begin{align*}
\sigma_k(t,u) &= (\delta_{k,i} \tilde{\sigma}_k(t,u_i))_{i \in \mathbb{Z}}, \quad q_k(t,u, y_k) = (q_{k,i}(t,u_i, y_k))_{i \in \mathbb{Z}}, \quad \text{for all } u \in \ell^2,
\end{align*}
\begin{align*}
F(u, v) &= (u_i^{2p} v_i)_{i \in \mathbb{Z}}, \quad G(u) = (u_i^{2p+1})_{i \in \mathbb{Z}}, \quad \text{for } u, v \in \ell^2.
\end{align*}In this paper, we will make use of the following inequalities:
\begin{align}\label{b0}
|\mathscr{X}^r - \mathscr{Y}^r| &\leq C_r |\mathscr{X} - \mathscr{Y}| \, |\mathscr{X}^{\,r-1} + \mathscr{Y}^{\,r-1}|,
\quad \forall \mathscr{X}, \mathscr{Y} \in \mathbb{R}, \, p \geq 1,
\end{align}and
\begin{align}\label{b00}
& b_1 b_2 \mathscr{X}^{2p+1} \mathscr{Y}
 - b_2^2 \mathscr{X}^{2p+2}
 - b_1^2 \mathscr{X}^{2p} \mathscr{Y}^2
 + b_2 b_1 \mathscr{X}^{2p+1} \mathscr{Y} \notag \\
=& 2 b_1 b_2 \mathscr{X}^{2p+1} \mathscr{Y}
 - \mathscr{X}^{2p} \big(b_2^2 \mathscr{X}^2 + b_1^2 \mathscr{Y}^2\big)
 \leq 0,
 \quad \forall \mathscr{X}, \mathscr{Y} \in \mathbb{R}, \, p \geq 1.
\end{align}According to \eqref{b0}, we have %and Young's inequality, we have
\begin{align}\label{b4}
& \|F(u_1, v_1) - F(u_2, v_2)\|^2 \leq c_1(n) \left( \|u_1 - u_2\|^2 + \|v_1 - v_2\|^2 \right), \nonumber \\
& |\langle F(u_1, v_1) - F(u_2, v_2), u_1 - u_2 \rangle| \leq c_1(n) \left( \|u_1 - u_2\|^2 + \|v_1 - v_2\|^2 \right), \notag \\
& |\langle F(u_1, v_1) - F(u_2, v_2), v_1 - v_2 \rangle| \leq c_1(n) \left( \|u_1 - u_2\|^2 + \|v_1 - v_2\|^2 \right),
\end{align}and
\begin{align}\label{b5}
& \|G(u) - G(v)\|^2 \leq c_2(n) \|u - v\|^2, \nonumber \\
& |\langle G(u_1) - G(u_2), u_1 - u_2 \rangle| \leq c_2(n) \|u_1 - u_2\|^2, \notag \\
& |\langle G(u_1) - G(u_2), v_1 - v_2 \rangle| \leq c_2(n) \left( \|u_1 - u_2\|^2 + \|v_1 - v_2\|^2 \right),
\end{align}where $u$, $v$, $u_1$, $u_2$, $v_1$, $v_2 \in \ell^2$, and $\|u\| \leq n$, $\|u_1\| \leq n$, $\|u_2\| \leq n$, $\|v\| \leq n$, $\|v_1\| \leq n$, $\|v_2\| \leq n$, see \cite{40} for the proof.  For any s $\in\mathbb{R}$
and $k\in\mathbb{N}$, we assume that there exist  $\alpha>0$ such that
\begin{align}\label{yhb1}
|\sigma_{k,i}(t,s)|\vee \int_{|y_k|<1}|q_{k,i}(t,s)|\nu_k\mathrm{d}y_k\leq\alpha\delta_{k,i}(t)(1+|s|).
\end{align}Indeed, for $C_k(n) > 0$, by \eqref{yhb1} and $\sigma_k$, $q_k$ are locally Lipschits continuous we have
\begin{align}\label{b6}
\sum_{k=1}^\infty \|\sigma_k(t,u_1) - \sigma_k(t,u_2)\|^2 \vee
\sum_{k=1}^\infty \int_{|y_k|<1} \|q_k(t,u_1, y_k) - q_k(t,u_2, y_k)\|^2 \nu_k(\mathrm{d}y_k)
\leq C_k(n) \|u_1 - u_2\|^2,
\end{align}
\begin{align}\label{b7}
\sum_{k=1}^\infty \|\sigma_k(t,u)\|^2 \vee \sum_{k=1}^\infty \int_{|y_k|<1} \|q_k(t,u, y_k)\|^2 \nu_k(\mathrm{d}y_k)
\leq 2\alpha^2 \|\delta(t)\|^2 (1 + \|u\|^2),
\end{align}where $\|\delta\|^2 := \sum\limits_{k \in \mathbb{N}} \sum\limits_{i \in \mathbb{Z}} |\delta_{k,i}|^2 $. In addition, we assume that
\begin{align}\label{bc0}
&\int_{-\infty}^{\tau } e^{\varpi t}\mathbb{E} \left( \|\delta(t)\|^2 + \sum_{k=1}^\infty \|\kappa_k(t)\|^2 + \sum_{k=1}^\infty \|h_k(t)\|^2 \right) \mathrm{d}t\notag\\
&\quad+\int_{-\infty}^{\tau }e^{\varpi t}\left(\|f_1(t)\|^2 + \|f_2(t)\|^2\right)\,\mathrm{d}t < \infty,\ \ \forall\tau \in \mathbb{R},
\end{align}where $\varpi>0$, see \eqref{xs}.

To express the abstract system \eqref{a1}-\eqref{a2} more concisely in the space $\ell^2 \times \ell^2$ (for $t \geq \tau$, $t, \tau \in \mathbb{R}$), we use the following system of equations:
\begin{align}\label{b8}
\left\{
\begin{aligned}
\mathrm{d}u(t) &= \left( -d_1 A u(t) - a_1 u(t) + b_1 F(u(t), v(t)) - b_2 G(u(t)) + f_1(t) \right) \mathrm{d}t \\
&\quad + \epsilon_1 \sum_{k=1}^\infty \left( h_k(t) + \sigma_k(t,u(t)) \right) \mathrm{d}W_k(t) \\
&\quad + \epsilon_2 \sum_{k=1}^\infty \left( \kappa_k(t) + q_k(t,u(t-), y_k) \right) \tilde{L}_k(\mathrm{d}t, \mathrm{d}y_k), \\
\mathrm{d}v(t) &= \left( -d_2 A v(t) - a_2 v(t) - b_1 F(u(t), v(t)) + b_2 G(u(t)) + f_2(t) \right) \mathrm{d}t \\
&\quad + \gamma_1 \sum_{k=1}^\infty \left( h_k(t) + \sigma_k(t,v(t)) \right) \mathrm{d}W_k(t) \\
&\quad + \gamma_2 \sum_{k=1}^\infty \left( \kappa_k(t) + q_k(t,v(t-), y_k) \right) \tilde{L}_k(\mathrm{d}t, \mathrm{d}y_k),
\end{aligned}
\right.
\end{align}
with the initial conditions
\begin{align}\label{bc1}
u(\tau) = u_\tau, \quad v(\tau) = v_\tau.
\end{align}As mentioned in reference \cite{LGF}, an $\ell^2\times\ell^2$-valued $c\grave{a}dl\grave{a}g$ process $\varphi(t)=(u(t),v(t))$ is
called a solution of system \eqref{b8}-\eqref{bc1} if $\mathcal{F}_{\tau}$-measurable initial conditions $\varphi(\tau)=(u_\tau,v_\tau)\in L^2(\Omega,\ell^2\times\ell^2)$ for any $t>\tau$ and for almost surely $\omega\in\Omega$,
\begin{align}\label{b9}
\left\{
\begin{aligned}
u(t)&=u_\tau+\int_\tau^t\big(-d_1Au(s)-a_1u(s)+b_1F(u(s),v(s))-b_2G(u(s))+f_1(s) \big)\mathrm{d}s\\
&\quad+\epsilon_1 \sum_{k=1}^\infty\int_\tau^t\big( h_k(s)+\sigma_k(t,u(s))\big)\mathrm{d}W_k(s)\\
&\quad+\epsilon_2\sum_{k=1}^\infty\int_\tau^t\int_{|y_k|<1}\big(\kappa_k(s)
+q_k(t,u(s-),y_k)\big)\tilde{L}_k(\mathrm{d}s,\mathrm{d}y_k),\\
v(t)&=v_\tau+\int_\tau^t\big(-d_2Av(s)-a_2v(s)-b_1F(u(s),v(s))+b_2G(u(s))+f_2(s) \big)\mathrm{d}s\\
&\quad+\gamma_1 \sum_{k=1}^\infty\int_\tau^t\big( h_k(s)+\sigma_k(t,v(s))\big)\mathrm{d}W_k(s)\\
&\quad+\gamma_2\sum_{k=1}^\infty\int_\tau^t\int_{|y_k|<1}\big(\kappa_k(s)
+q_k(t,v(s-),y_k)\big)\tilde{L}_k(\mathrm{d}s,\mathrm{d}y_k).\\
\end{aligned}
\right.
\end{align}\textbf{\textcolor{red}{Therefore, }}under assumptions \eqref{b6}-\eqref{bc0}, the existence and uniqueness of solutions to system \eqref{b8}-\eqref{bc1} can be obtained.
\section{Pullback measure attractors} %of \eqref{a1}-\eqref{a2} on $(\mathcal{P}_{\kappa}(\mathbf{X}),
%d_{\mathcal{P}(\mathbf{X})})$}

For clarity, we first review the fundamental concepts of pullback measure attractors and then outline their structural properties(see, e.g., \cite{LID1}).

We denote by $X$ a separable Banach space with norm $\|\cdot\|_X$. Let $C_b(X)$ be
the space of bounded continuous functions $\phi : X \to \mathbb{R}$ endowed with the norm
\begin{align*}
\|\phi\|_\infty = \sup_{x \in X} |\phi(x)|.
\end{align*}Denote by $L_b(X)$ the space of bounded Lipschitz functions on $X$ which consists of all functions $\phi \in C_b(X)$ such that
\begin{align*}
\mathrm{Lip}(\phi) := \sup_{x_1, x_2 \in X, x_1 \neq x_2} \frac{|\phi(x_1) - \phi(x_2)|}{\|x_1 - x_2\|_X} < \infty.
\end{align*}The space $L_b(X)$ is endowed with the norm
\begin{align*}
\|\phi\|_L = \|\phi\|_\infty + \mathrm{Lip}(\phi).
\end{align*}Let $\mathcal{P}(X)$ be the set of probability measures on $(X, \mathcal{B}(X))$, where $\mathcal{B}(X)$ is the Borel $\sigma$-algebra
of $X$. Given $\phi \in C_b(X)$ and $\mu \in \mathcal{P}(X)$, we write
\begin{align*}
\langle\phi, \mu\rangle = \int_X \phi(x) \mu(\mathrm{d}x).
\end{align*}Recall that a sequence $\{\mu_n\}_{n=1}^\infty \subset \mathcal{P}(X)$ is weakly convergent to $\mu \in \mathcal{P}(X)$ if for every $\phi \in C_b(X)$,
\begin{align*}
\lim_{n \to \infty} \langle\phi, \mu_n\rangle = \langle\phi, \mu\rangle.
\end{align*}Define a metric on $\mathcal{P}(X)$ by
\begin{align*}
d_{\mathcal{P}(X)} (\mu_1, \mu_2) = \sup_{\phi \in L_b(X), \|\phi\|_L \leq 1} |(\phi, \mu_1) - (\phi, \mu_2)|, \quad \forall \mu_1, \mu_2 \in \mathcal{P}(X).
\end{align*}Then $(\mathcal{P}(X), d_{\mathcal{P}(X)})$ is a polish space. Moreover, a sequence $\{\mu_n\}_{n=1}^\infty \subset \mathcal{P}(X)$ converges to $\mu$ in $(\mathcal{P}(X), d_{\mathcal{P}(X)})$ if and only if $\{\mu_n\}_{n=1}^\infty$ converges to $\mu$ weakly.

Given $p > 0$, let $\mathcal{P}_p(X)$ be the subset of $\mathcal{P}(X)$ as defined by
\begin{align*}
\mathcal{P}_p(X) = \left\{\mu \in \mathcal{P}(X) : \int_X \|x\|_X^p \mu(\mathrm{d}x) < \infty \right\}.
\end{align*}Then $(\mathcal{P}_p(X), d_{\mathcal{P}(X)})$ is also a metric space. Without confusion, we also denote $(\mathcal{P}_p(X), d_{\mathcal{P}(X)})$ by $(\mathcal{P}_p(X), d_{\mathcal{P}_p(X)})$. Given $r > 0$, denote by
\begin{align*}
B_{\mathcal{P}_p(X)}(r) = \left\{\mu \in \mathcal{P}_p(X) : \left(\int_X \|x\|_X^p \mu(\mathrm{d}x)\right)^{\frac{1}{p}} \leq r \right\}.
\end{align*}Recall that the Hausdorff semi-metric between subsets of $\mathcal{P}_p(X)$ is given by
\begin{align*}
d_{\mathcal{P}_p(X)}(Y, Z) = \sup_{y \in Y} \inf_{z \in Z} d_{\mathcal{P}_p(X)}(y, z), \quad Y, Z \subset \mathcal{P}_p(X), \quad Y, Z \neq \emptyset.
\end{align*}If $\epsilon > 0$ and $B \subset \mathcal{P}_p(X)$, then the open $\epsilon$-neighborhood of $B$ in $\mathcal{P}_p(X)$ is defined by
\begin{align*}
\mathcal{N}_\epsilon(B) = \{\mu \in \mathcal{P}_p(X) : d_{\mathcal{P}_p(X)}(\mu, B) < \epsilon\}.
\end{align*}
\begin{definition}A family $\mathcal{S} = \{S(t, \tau) : t \in \mathbb{R}^+, \tau \in \mathbb{R}\}$ of mappings from $\mathcal{P}_p(X)$ to $\mathcal{P}_p(X)$ is
called a continuous non-autonomous dynamical system on $\mathcal{P}_p(X)$, if for all $\tau \in \mathbb{R}$ and $t, s \in \mathbb{R}^+$,
the following conditions are satisfied:
\begin{enumerate}[(a)]
  \item  $S(0, \tau) = I_{\mathcal{P}_p(X)}$, $\text{where } I_{\mathcal{P}_p(X)} \text{ is the identity operator on } \mathcal{P}_p(X);$ \\
  \item  $S(t + s, \tau) = S(t, s + \tau) \circ S(s, \tau)$; \\
  \item   $S(t, \tau) : \mathcal{P}_p(X) \to \mathcal{P}_p(X) \text{ is continuous}$.
\end{enumerate}
\end{definition}
\begin{definition}A set $D \subset \mathcal{P}_p(X)$ is called a bounded subset if there is $r > 0$ such that $D \subset B_{\mathcal{P}_p(X)}(r)$.
\end{definition}

In the sequel, we denote by $\mathcal{D}$ a collection of some families of nonempty subsets of $\mathcal{P}_p(X)$
parametrized by $\tau \in \mathbb{R}$; that is,
\begin{align*}
\mathcal{D} = \{D = \{D(\tau) \subset \mathcal{P}_p(X) : D(\tau) \neq \emptyset, \tau \in \mathbb{R}\} : D \text{ satisfies some conditions}\}.
\end{align*}
\begin{definition}A collection $\mathcal{D}$ of some families of nonempty subsets of $\mathcal{P}_p(X)$ is said to be neighborhood-closed if for each $D = \{D(\tau) : \tau \in \mathbb{R}\} \in \mathcal{D}$, there exists a positive number $\epsilon$
depending on $D$ such that the family
\begin{align*}
\{B(\tau) : B(\tau) \text{ is a nonempty subset of } \mathcal{N}_\epsilon(D(\tau)), \, \forall \tau \in \mathbb{R}\}
\end{align*}also belongs to $\mathcal{D}$.
\end{definition}
Note that the neighborhood closedness of $\mathcal{D}$ implies for each $D \in \mathcal{D}$,
\begin{align}\label{z20}
\tilde{D} = \{\tilde{D}(\tau) : \tilde{D}(\tau) \neq \emptyset \,\  \text{and}\  \, \tilde{D}(\tau) \subseteq D(\tau), \, \tau \in \mathbb{R}\} \in \mathcal{D}.
\end{align}A collection $\mathcal{D}$ satisfying \eqref{z20} is said to be inclusion-closed in the literature.
\begin{definition}
A family $K = \{K(\tau) : \tau \in \mathbb{R}\} \in \mathcal{D}$ is called a $\mathcal{D}$-pullback absorbing set for $S$ if for each $\tau \in \mathbb{R}$ and every $D \in \mathcal{D}$, there exists $T = T(\tau, D) > 0$ such that
\begin{align*}
S(t, \tau - t) D(\tau - t) \subset K(\tau), \quad \text{for all }\  t \geq T.
\end{align*}
\end{definition}
\begin{definition}The non-autonomous dynamical system $S$ is said to be $\mathcal{D}$-pullback asymptotically compact in $\mathcal{P}_p(X)$
if for each $\tau \in \mathbb{R}$, $\{S(t_n, \tau - t_n) \mu_n\}_{n=1}^\infty$ has a convergent subsequence in $\mathcal{P}_p(X)$ whenever $t_n \to \infty$
and $\mu_n \in D(\tau - t_n)$ with $D \in \mathcal{D}$.
\end{definition}
\begin{definition}
A family $\mathcal{A} = \{\mathcal{A}(\tau) : \tau \in \mathbb{R}\} \in \mathcal{D}$ is called a $\mathcal{D}$-pullback measure attractor for $S$
if the following conditions are satisfied:
\begin{enumerate}[(i)]
  \item \, $\mathcal{A}(\tau) \text{ is compact in } \mathcal{P}_p(X) \text{ for each } \tau \in \mathbb{R}$;
  \item \,  $\mathcal{A} \text{ is invariant, that is, } S(t, \tau) A(\tau ) = A(\tau + t), \, \text{for all } \tau \in \mathbb{R} \text{ and } t \in \mathbb{R}^+$;
  \item \,  $\mathcal{A} \text{ attracts every set in } \mathcal{D}, \text{ that is, for each } D = \{D(\tau) : \tau \in \mathbb{R}\} \in \mathcal{D}$,
\begin{align}
 \lim_{t \to \infty} d(S(t, \tau - t) D(\tau - t), A(\tau)) = 0.
\end{align}
\end{enumerate}
\end{definition}

\begin{definition}
A mapping $\psi : \mathbb{R} \times \mathbb{R} \to \mathcal{P}_p(X)$ is called a complete orbit of $S$ if for every $s \in \mathbb{R}, t \in \mathbb{R}^+$ and $\tau \in \mathbb{R}$, the following holds:
\begin{align}
S(t, s + \tau) \psi(s, \tau) = \psi(t + s, \tau).
\end{align}
\end{definition}If, in addition, there exists $D = \{D(\tau) : \tau \in \mathbb{R}\} \in \mathcal{D}$ such that $\psi(t, \tau)$ belongs to $D(t + \tau)$ for
every $t \in \mathbb{R}$ and $\tau \in \mathbb{R}$, then $\psi$ is called a $\mathcal{D}$-complete orbit of $S$.
\begin{definition}
A mapping $\xi : \mathbb{R} \to \mathcal{P}_p(X)$ is called a complete solution of $S$ if for every $t \in \mathbb{R}^+$
and $\tau \in \mathbb{R}$, the following holds:
\begin{align*}
S(t, \tau) \xi(\tau) = \xi(t + \tau).
\end{align*}
If, in addition, there exists $D = \{D(\tau) : \tau \in \mathbb{R}\} \in \mathcal{D}$ such that $\xi(\tau)$ belongs to $D(\tau)$ for every
$\tau \in \mathbb{R}$, then $\xi$ is called a $\mathcal{D}$-complete solution of $S$.
\end{definition}
\begin{definition}
For each $D = \{D(\tau) : \tau \in \mathbb{R}\} \in \mathcal{D}$ and $\tau \in \mathbb{R}$, the pullback $\omega$-limit set of $D$ at $\tau$
is defined by
\begin{align*}
\omega(D, \tau) = \bigcap_{s \geq 0} \overline{\bigcup_{t \geq s} S(t, \tau - t) D(\tau - t)},
\end{align*}that is,
\begin{align*}
\omega(D, \tau) = \limsup_{t \to \infty} S(t, \tau - t) D(\tau - t).
\end{align*}
\end{definition}
\begin{align*}
\omega(D, \tau) = \left\{ \nu \in \mathcal{P}_p(X) : \text{there exist } t_n \to \infty, \, \mu_n \in D(\tau - t_n) \text{ such that }
\nu = \lim_{n \to \infty} S(t_n, \tau - t_n) \mu_n \right\}.
\end{align*}

Based on above notation, from Proposition 3.6 in \cite{16}, we have the following criterion for the
existence and uniqueness of $\mathcal{D}$-pullback measure attractors.

\begin{proposition}\label{h19}
Let $\mathcal{D}$ be a   collection of families of subsets of $\mathcal{P}_p(X)$
and $S$ be a continuous non-autonomous dynamical system on $\mathcal{P}_p(X)$. Then $S$ has a unique
$\mathcal{D}$-pullback measure attractor $\mathcal{A}$ in $\mathcal{P}_p(X)$ if $S$ has a closed $\mathcal{D}$-pullback absorbing
set $K \in \mathcal{D}$ and $S$ is $\mathcal{D}$-pullback asymptotically compact in $\mathcal{P}_p(X)$. The $\mathcal{D}$-pullback measure
attractor $\mathcal{A}$ is given by, for each $\tau \in \mathbb{R}$,
\begin{align*}
\mathcal{A}(\tau) = \omega(K, \tau) = \{\psi(0, \tau) : \psi \text{ is a } \mathcal{D}\text{-complete orbit of } S\} \\
= \{\xi(\tau) : \xi \text{ is a } \mathcal{D}\text{-complete solution of } S\}.
\end{align*}
\end{proposition} To emphasize the role of the initial time and initial values, we denote by $\varphi(t, \tau,\xi )$ the solution of \eqref{b8}-\eqref{bc1} with initial conditions $\varphi(\tau) = \xi= (u_\tau,v_\tau )\in L_{\mathcal{F}_\tau}^2(\Omega, \ell^2\times\ell^2)$ and $\varphi\in L_{\mathcal{F}_\tau}^2(\Omega, \ell^2\times\ell^2)$. Give a subset $E$ of $\mathcal{P}_2(\ell^2\times\ell^2)$, define
\begin{align*}
\|E\|_{\mathcal{P}_2(\ell^2\times\ell^2)} = \inf \left\{ r > 0 : \sup_{\mu \in E} \left( \int_{\ell^2\times\ell^2} \|z\|_{\ell^2\times\ell^2}^2 \mu(dz) \right)^{\frac{1}{2}} \leq r \right\},
\end{align*}
with the convention that $\inf \emptyset = \infty$. If $E$ is a bounded subset of $\mathcal{P}_2(\ell^2\times\ell^2)$, then $\|E\|_{\mathcal{P}_2(\ell^2\times\ell^2)} < \infty$. Let $\mathcal{D}$ be the collection of families of bounded nonempty subsets of $\mathcal{P}_2(\ell^2\times\ell^2)$ as given by
\begin{align*}
\mathcal{D} = \left\{ D = \{D(\tau) \subset \mathcal{P}_2(\ell^2\times\ell^2) : \emptyset \neq D(\tau) \text{ bounded in } \mathcal{P}_2(\ell^2\times\ell^2), \, \tau \in \mathbb{R} \} \right\},
\end{align*}and
\begin{align}\label{temped}
\lim_{\tau \to -\infty} e^{\varpi \tau} \|D(\tau)\|_{\mathcal{P}_2(\ell^2\times\ell^2)}^2 = 0 ,
\end{align}
where $\varpi > 0$ defined later.
\section{Uniform moment estimates}
In this section, we derive uniform moment estimates of the solution of problem \eqref{b8}-\eqref{bc1} which are necessary for establishing the existence of pullback measure attractors. To this end, we use $\mathcal{L}(\varphi)$ to denote the distribution law of the random variable, and $C$ to represent any positive constants that may change from line to line. We further assume that
\begin{align}\label{xs}
\varpi = \inf_{t \in \mathbb{R}} \left\{ (a_1\wedge a_2)-\frac{1}{2}-4\alpha^2\|\delta(t)\|^2(b_2+b_1) \right\} > 0.
\end{align}Discuss uniform estimates of solutions of problem \eqref{b8}-\eqref{bc1} in $L^2(\Omega, \ell^2\times\ell^2)$.
\begin{lemma}\label{yingli41}
Suppose \eqref{b6}-\eqref{bc0} and \eqref{xs} hold. Then for every $\tau \in \mathbb{R}, t\geq0$ and random data $\xi=(u_\tau,v_\tau) \in L_{\mathcal{F}_{\tau - t}}^2(\Omega, \ell^2\times\ell^2)$, and $0 < \epsilon_1,\epsilon_2,\gamma_1,\gamma_2 \leq 1$, the solution $\varphi$ of \eqref{b8}-\eqref{bc1} satisfies
\begin{align}
\mathbb{E} \left( \|\varphi(\tau, \tau - t, \xi)\|^2 \right) + \int_{\tau - t}^\tau e^{-\varpi(\tau - s)} \mathbb{E} \left( \|\varphi(s, \tau - t, \xi))\|^2 \right) \mathrm{d}s \leq C\Big(e^{-\varpi t}\mathbb{E} \left(\|\xi\|^2_{\ell^2\times\ell^2}  \right)+R(\tau)\Big),
\end{align}
where
\begin{align}\label{op}
R(\tau)= &   \int_{-\infty}^\tau e^{-\varpi(\tau - s)} \left( \sum_{k=1}^\infty\|\kappa_k(s)\|^2+\sum_{k=1}^\infty\|h_k(s)\|^2+\|f_2(s)\|^2+\|f_1(s)\|^2 \right) \mathrm{d}s\notag\\
&+ \int_{-\infty}^\tau e^{-\varpi(\tau - s)} \|\delta(s)\|^2\,\mathrm{d}s
\end{align}
with $C > 0$ being a constant independent of $\tau$, $\epsilon_1,\epsilon_2,\gamma_1,$  and $ \gamma_2.$
\end{lemma}
\begin{proof}
By \eqref{b8} and the It\^{o} formula, we obtain
\begin{align}\label{51}
&\mathrm{d}\Big[b_2\|u(s)\|^2\Big]+\Big[2d_1b_2\|Bu(s)\|^2
+2a_1b_2\|u(s)\|^2+2b^2_2\langle G(u(s)),u(s)\rangle
\Big]\,\mathrm{d}s\notag\\
=&2b_2\langle f_1(s),u(s)\rangle\,\mathrm{d}s
+2b_1b_2
\langle F(u(s),v(s)),u(s)\rangle\,\mathrm{d}s\notag\\
&+b_2\epsilon_1^2\sum_{k=1}^{\infty}
\|h_k(s)+\sigma_k(t,u(s))\|^2\,\mathrm{d}s+2b_2\epsilon_1\sum_{k=1}^{\infty}
\langle
h_k(s)+\sigma_k(t,u(s)),u(s)
\rangle\, \mathrm{d}W_k(s)\notag\\
&+b_2\sum_{k=1}^{\infty}\int_{|y_k|<1}\Big( \|
u(s-)+\epsilon_2(q_k(t,u(s-),y_k)+\kappa_k(s))
\|^2-\|u(s-)\|^2
\Big)\tilde{L}_k(\mathrm{d}s,\mathrm{d}y_k)\notag\\
&+b_2\epsilon_2^2\sum_{k=1}^{\infty}\int_{|y_k|<1}
\|q_k(t,u(s),y_k)+\kappa_k(s)\|^2\nu_k(\mathrm{d}y_k)\,\mathrm{d}s,
\end{align}and
\begin{align}\label{b52}
&\mathrm{d}\,\Big[b_1\|v(s)\|^2\Big]+\Big[2d_2b_1\|Bv(s)
\|^2+2a_2b_1\|v(s)\|^2+2b_1^2\big(F(u(s),v(s)),
v(s)\big)\Big]\,\mathrm{d}s\nonumber\\
=&2b_1\langle f_2(s),v_n(s)\rangle\,\mathrm{d}s +2b_2b_1\big\langle G^n(u_n(s)),v_{n}(s)\big\rangle\,\mathrm{d}s
\nonumber\\
&+b_1\gamma^2_1\sum_{k=1}^\infty \|h_k(s)+\sigma_k(t,v(s))\|^2\,\mathrm{d}s
+2b_1\gamma_1\sum_{k=1}^\infty\big\langle h_k(s)+\sigma_k(t,v(s)),v(s)\big\rangle \mathrm{d}W_k(s)\notag\\
&+b_1\sum_{k=1}^{\infty}\int_{|y_k|<1}\Big(
\|
v(s-)+\gamma_2(q_k(t,v(s-),y_k)+\kappa_k(s))
\|^2-\|v(s-)\|^2
\Big)\tilde{L}_k(\mathrm{d}s,\mathrm{d}y_k)\notag\\
&+b_1\gamma_2^2\sum_{k=1}^{\infty}\int_{|y_k|<1}
\|q_k(t,v(s),y_k)+\kappa_k(s)\|^2\nu_k(\mathrm{d}y_k)\,\mathrm{d}s.
\end{align}Taking expectation of \eqref{51}+\eqref{b52}, we get
\begin{align}\label{b53}
&\frac{\mathrm{d}}{\mathrm{d}s}\,\mathbb{E}\Big[\|b_2u(s)\|^2+\|b_1v(s)\|^2\Big]=
\mathbb{E}\Big[-2d_1b_2\|Bu(s)\|^2-2a_1b_2\|u(s)\|^2-2b^2_2\langle G(u(s)),u(s)\rangle
\Big]\notag\\
&\quad\quad\quad+\mathbb{E}\Big[-2d_2b_1\|Bv(s)
\|^2-2a_2b_1\|v(s)\|^2-2b_1^2\big(F(u(s),v(s)),
v(s)\big)\Big]\notag\\
&\quad\quad\quad+2b_2\mathbb{E}\Big[\langle f_1(s),u(s)\rangle\Big]
+2b_1b_2\mathbb{E}\Big[\langle F(u(s),v(s)),u(s)\rangle\Big]\notag\\
&\quad\quad\quad+b_2\epsilon_1^2\sum_{k=1}^{\infty}
\mathbb{E}\Big[\|h_k(s)+\sigma_k(t,u(s))\|^2\Big]+b_1\gamma^2_1\sum_{k=1}^\infty
\mathbb{E}\Big[\|h_k(s)+\sigma_k(t,v(s))\|^2\Big]\notag\\
&\quad\quad\quad+2b_1\mathbb{E}\Big[\langle f_2(s),v_n(s)\rangle\Big] +2b_2b_1\mathbb{E}\Big[\langle G(u(s)),v(s)\rangle\Big]\notag\\
&\quad\quad\quad+b_1\gamma_2^2\sum_{k=1}^{\infty}\int_{|y_k|<1}\mathbb{E}\Big[
\|q_k(t,v(s),y_k)+\kappa_k(s)\|^2\nu_k(\mathrm{d}y_k)\Big]\notag\\
&\quad\quad\quad+b_2\epsilon_2^2\sum_{k=1}^{\infty}\int_{|y_k|<1}\mathbb{E}\Big[
\|q_k(t,u(s),y_k)+\kappa_k(s)\|^2\nu_k(\mathrm{d}y_k)\Big].
\end{align}By employing inequalities \eqref{b2}-\eqref{b00},  we deduce that
\begin{align*}
b_2 b_1 \langle F(u,v), u \rangle - b_2^2 \langle G(u), u \rangle - b_1^2 \langle F(u,v), v \rangle + b_2 b_1 \langle G(u), v \rangle \leq 0.
\end{align*}\textbf{\textcolor{red}{And}}
\begin{small}\begin{align}\label{b61}
&\frac{\mathrm{d}}{\mathrm{d}s}\mathbb{E}\Big[b_1\|v(s)\|^2+b_2\|u(s)\|^2\Big]\notag\\
\leq&-2\varsigma\mathbb{E}\Big(b_2\|u(s)\|^2+b_1\|v(s)\|^2\Big)+
\Bigl(4b_2\epsilon_1^2+4b_1\gamma_1^2+4b_2\epsilon_2^2+4b_1\gamma_2^2\Bigl)\alpha^2\|\delta(s)\|^2\notag\\
&+(1+2b_2\epsilon_1^2+
2b_1\gamma_1^2+2b_1\gamma_2^2+2b_2\epsilon_2^2)
\mathbb{E}\Big[\sum_{k=1}^\infty\|\kappa_k(s)\|^2+\sum_{k=1}^\infty\|h_k(s)\|^2+b_1\|f_2(s)\|^2+b_2\|f_1(s)\|^2\Big]\notag\\
&+\big(1+4\alpha^2\|\delta\|^2(b_2\epsilon_1^2+b_2\epsilon_2^2+b_1\gamma_1^2+b_1\gamma_2^2)\big)
\mathbb{E}\Big[b_2\|u(s)\|^2+b_1\|v(s)\|^2\Big]\textbf{\textcolor{red}{,}}
\end{align}\end{small}where $\varsigma=a_1\wedge a_2$. Utilizing the aforementioned inequality, we have
\begin{small}\begin{align}\label{b61}
&\frac{\mathrm{d}}{\mathrm{d}s}\mathbb{E}\Big[b_1\|v(s)\|^2+b_2\|u(s)\|^2\Big]\notag\\
\leq&-2\varpi\mathbb{E}\Big(b_2\|u(s)\|^2+b_1\|v(s)\|^2\Big)+
\Bigl(4b_2\epsilon_1^2+4b_1\gamma_1^2+4b_2\epsilon_2^2+4b_1\gamma_2^2\Bigl)\alpha^2\|\delta(s)\|^2\notag\\
&+(1+2b_2\epsilon_1^2+
2b_1\gamma_1^2+2b_1\gamma_2^2+2b_2\epsilon_2^2)
\mathbb{E}\Big[\sum_{k=1}^\infty\|\kappa_k(s)\|^2+\sum_{k=1}^\infty\|h_k(s)\|^2+b_1\|f_2(s)\|^2+b_2\|f_1(s)\|^2\Big].
\end{align}\end{small}Multiplying \eqref{b61} by $e^{\varpi t}$ and then integrating the resulting inequality on $(\tau - t, \tau)$ with $t \in \mathbb{R}^+$, we obtain
\begin{align}\label{nl}
&\mathbb{E} \left( b_2\|u(\tau, \tau - t, u_\tau)\|^2+b_1\|v(\tau, \tau - t, v_\tau)\|^2 \right)\notag\\
&+ \varpi \int_{\tau - t}^\tau e^{-\varpi(\tau - s)} \mathbb{E} \left( b_2\|u(s, \tau - t, u_\tau)\|^2+b_1\|v(s, \tau - t, v_\tau)\|^2 \right) \mathrm{d}s \nonumber \\
\leq &\mathbb{E} \left(\|\xi\|^2_{\ell^2\times\ell^2} \right) e^{-\varpi t} + C \int_{\tau - t}^\tau e^{-\varpi(\tau - s)} \left( \sum_{k=1}^\infty\|\kappa_k(s)\|^2+\sum_{k=1}^\infty\|h_k(s)\|^2+b_1\|f_2(s)\|^2+b_2\|f_1(s)\|^2 \right) \mathrm{d}s\notag\\
&+\Bigl(4b_2\epsilon_1^2+4b_1\gamma_1^2+4b_2\epsilon_2^2+4b_1\gamma_2^2\Bigl)\alpha^2
\int_{\tau-t}^\tau e^{-\varpi(\tau - s)}\|\delta(s)\|^2\,\mathrm{d}s\notag\\
\leq& \mathbb{E} \left(\|\xi\|^2_{\ell^2\times\ell^2} \right) e^{-\varpi t} + C \int_{-\infty}^\tau e^{-\varpi(\tau - s)} \left( \sum_{k=1}^\infty\|\kappa_k(s)\|^2+\sum_{k=1}^\infty\|h_k(s)\|^2+b_1\|f_2(s)\|^2+b_2\|f_1(s)\|^2 \right) \mathrm{d}s\notag\\
&+\Bigl(4b_2\epsilon_1^2+4b_1\gamma_1^2+4b_2\epsilon_2^2+4b_1\gamma_2^2\Bigl)\alpha^2
\int_{-\infty}^\tau e^{-\varpi(\tau - s)}\|\delta(s)\|^2\,\mathrm{d}s,
\end{align} which   concludes the proof.
\end{proof}
Next, we derive uniform estimates on the tails of the solutions of \eqref{b8}-\eqref{bc1} which are crucial for establishing the $\mathcal{D}$-pullback asymptotic compactness in $\mathcal{P}_2(\ell^2\times\ell^2)$ of the family of probability distributions of the solutions.
\begin{lemma}\label{z42}
Suppose \eqref{b6}-\eqref{bc0} and \eqref{xs} hold. Then for every $\eta > 0$  and $\tau \in \mathbb{R}$, there exists  $N = N(  \tau, \eta) \in \mathbb{N}$ such that for all $0 < \epsilon_1,\epsilon_2,\gamma_1,\gamma_2 \leq 1$, $t \geq 0$ and $n \geq N$, the solution $\varphi$ of \eqref{b8}-\eqref{bc1} satisfies
\begin{align}\label{weba1}
\sum_{|i| \geq n} \mathbb{E}(\left|u_i(\tau, \tau - t, u_\tau)\right|^2+\left|v_i(\tau, \tau - t, v_\tau)\right|^2) \leq  Ce^{-\varpi t}\mathbb{E} \left(\|\xi\|^2_{\ell^2\times\ell^2} \right)+ \eta,
\end{align}
when $\xi=( u_\tau, v_\tau) \in L_{\mathcal{F}_{\tau - t}}^2(\Omega, \ell^2\times\ell^2)$.
\end{lemma}

\begin{proof}
Let $\theta : \mathbb{R} \to \mathbb{R}$ be a smooth function such that $0 \leq \theta(s) \leq 1$ for all $s \in \mathbb{R}$ and
\begin{align*}
\theta(s) = 0, \quad \text{for } |s| \leq 1, \quad \text{and } \theta(s) = 1, \quad \text{for } |s| \geq 2.
\end{align*}
Given $n \in \mathbb{N}$, denote by $\theta_n = (\theta(\frac{i}{n}))_{i \in \mathbb{Z}}$ and $\theta_n u = (\theta(\frac{i}{n}) u_i)_{i \in \mathbb{Z}}$, $\theta_n v = (\theta(\frac{i}{n}) v_i)_{i \in \mathbb{Z}}$ for $u = (u_i)_{i \in \mathbb{Z}}$, $v = (v_i)_{i \in \mathbb{Z}}$. By \eqref{b8}, we obtain
\begin{align}\label{z22}
\left\{
\begin{aligned}
\mathrm{d}\theta_nu(t) &= \left( -d_1\theta_n A u(t) - a_1\theta_n u(t) + b_1\theta_n F(u(t), v(t)) - b_2\theta_n G(u(t)) + \theta_nf_1(t) \right) \mathrm{d}t \\
&\quad + \epsilon_1 \sum_{k=1}^\infty \left( \theta_nh_k(t) + \theta_n\sigma_k(t,u(t)) \right) \mathrm{d}W_k(t) \\
&\quad + \epsilon_2 \sum_{k=1}^\infty \left( \theta_n\kappa_k(t) + \theta_nq_k(t,u(t-), y_k) \right) \tilde{L}_k(\mathrm{d}t, \mathrm{d}y_k), \\
\mathrm{d}\theta_nv(t) &= \left( -d_2\theta_n A v(t) - a_2\theta_n v(t) - b_1\theta_n F(u(t), v(t)) + b_2\theta_n G(u(t)) + \theta_nf_2(t) \right) \mathrm{d}t \\
&\quad + \gamma_1 \sum_{k=1}^\infty \left( \theta_nh_k(t) + \theta_n\sigma_k(t,v(t)) \right) \mathrm{d}W_k(t) \\
&\quad + \gamma_2 \sum_{k=1}^\infty \left( \theta_n\kappa_k(t) + \theta_nq_k(t,v(t-), y_k) \right) \tilde{L}_k(\mathrm{d}t, \mathrm{d}y_k),
\end{aligned}
\right.
\end{align}From equation \eqref{z22} and It\^{o}'s formula, we respectively obtain the following results
\begin{align}\label{z23}
&\mathrm{d}\Big[b_2\|\theta_nu(s)\|^2\Big]+\Big[2d_1b_2\left\langle Bu(t),B\left(\theta_n^2u(t)\right)\right\rangle
+2a_1b_2\|\theta_nu(s)\|^2+2b^2_2\left\langle \theta_nG(u(s)),\theta_nu(s)\right\rangle
\Big]\,\mathrm{d}s\notag\\
=&2b_2\left\langle \theta_nf_1(s),\theta_nu(s)\right\rangle\,\mathrm{d}s
+2b_1b_2
\left\langle \theta_nF(u(s),v(s)),\theta_nu(s)\right\rangle\,\mathrm{d}s\notag\\
&+b_2\epsilon_1^2\sum_{k=1}^{\infty}
\|\theta_nh_k(s)+\theta_n\sigma_k(t,u(s))\|^2\,\mathrm{d}s+2b_2\epsilon_1\sum_{k=1}^{\infty}
\langle
\theta_nh_k(s)+\theta_n\sigma_k(t,u(s)),\theta_nu(s)
\rangle\, \mathrm{d}W_k(s)\notag\\
&+b_2\sum_{k=1}^{\infty}\int_{|y_k|<1}\Big( \|
\theta_nu(s-)+\epsilon_2(\theta_nq_k(t,u(s-),y_k)+\theta_n\kappa_k(s))
\|^2-\|\theta_nu(s-)\|^2
\Big)\tilde{L}_k(\mathrm{d}s,\mathrm{d}y_k)\notag\\
&+b_2\epsilon_2^2\sum_{k=1}^{\infty}\int_{|y_k|<1}
\|\theta_nq_k(t,u(s),y_k)+\theta_n\kappa_k(s)\|^2\nu_k(\mathrm{d}y_k)\,\mathrm{d}s,
\end{align}and
\begin{align}\label{b52}
&\mathrm{d}\,\Big[b_1\|\theta_nv(s)\|^2\Big]+\Big[2d_2b_1\left\langle Bv(s),B\left(\theta_n^2v(s)\right)\right\rangle+2a_2b_1\|\theta_nv(s)\|^2+2b_1^2\big(\theta_nF(u(s),v(s)),
\theta_nv(s)\big)\Big]\,\mathrm{d}s\nonumber\\
=&2b_1\langle \theta_nf_2(s),\theta_nv_n(s)\rangle\,\mathrm{d}s +2b_2b_1\big\langle \theta_nG^n(u_n(s)),\theta_nv_{n}(s)\big\rangle\,\mathrm{d}s
\nonumber\\
&+b_1\gamma^2_1\sum_{k=1}^\infty \|\theta_nh_k(s)+\theta_n\sigma_k(t,v(s))\|^2\,\mathrm{d}s
+2b_1\gamma_1\sum_{k=1}^\infty\big\langle \theta_nh_k(s)+\theta_n\sigma_k(t,v(s)),\theta_nv(s)\big\rangle \mathrm{d}W_k(s)\notag\\
&+b_1\sum_{k=1}^{\infty}\int_{|y_k|<1}\Big(
\|
\theta_nv(s-)+\gamma_2(\theta_nq_k(t,v(s-),y_k)+\theta_n\kappa_k(s))
\|^2-\|\theta_nv(s-)\|^2
\Big)\tilde{L}_k(\mathrm{d}s,\mathrm{d}y_k)\notag\\
&+b_1\gamma_2^2\sum_{k=1}^{\infty}\int_{|y_k|<1}
\|\theta_nq_k(t,v(s),y_k)+\theta_n\kappa_k(s)\|^2\nu_k(\mathrm{d}y_k)\,\mathrm{d}s.
\end{align}For the second term on the left-hand side of \eqref{b52}, we get
\begin{align}\label{z24}
&2 d_1 b_2 \mathbb{E} \left( \langle B v(t), B(\theta_n^2 v(t)) \rangle \right)\notag\\
=& 2 d_1 b_2 \mathbb{E} \left( \sum_{i \in \mathbb{Z}} \left( v_{i+1} - v_i \right)
\left( \theta_n^2 \left( i + 1 / n \right) v_{i+1} - \theta_n^2 \left( i / n \right) v_i \right) \right) \nonumber \\
=& 2 d_1 b_2 \mathbb{E} \left( \sum_{i \in \mathbb{Z}} \theta_n^2 \left( i + 1 / n \right) \left( v_{i+1} - v_i \right)^2 \right) + 2 d_1 b_2 \mathbb{E} \left(\sum_{i \in \mathbb{Z}} \left( \theta_n^2 \left( i + 1 / n \right) - \theta^2 \left( i / n \right) \right)
\left( v_{i+1} - v_i\right) v_i \right).
\end{align}For the last term of \eqref{z24}, we have
\begin{align}\label{z25}
&\Big| 2d_2 b_1 \mathbb{E} \Big( \sum_{i \in \mathbb{Z}} \big( \theta^2(i + 1 / n) - \theta^2(i / n) \big)
\big( v_{i+1} - u_i \big) v_i \Big) \Big|\notag\\
\leq& 2d_2 b_1 \mathbb{E} \Big( \sum_{i \in \mathbb{Z}} \big| \theta(i + 1 / n) + \theta(i / n) \big|
\big| \theta(i + 1 / n) - \theta(i / n) \big| \big| v_{i+1} - v_i \big|
\big| v_i\big| \Big) \nonumber \\
\leq &4d_2 b_1 \mathbb{E} \Big( \sum_{i \in \mathbb{Z}} \big| \theta(i + 1 / n) - \theta(i / n) \big|
\big| v_{i+1} - v_i \big| \big| v_i \big| \Big) \nonumber \\
\leq &\frac{4d_2 b_1}{n} \mathbb{E} \Big( \sum_{i \in \mathbb{Z}} \big| \theta'(\eta_i) \big|
\big( \big| v_{i+1} \big| + \big| v_i \big| \big) \big| v_i \big| \Big) \nonumber \\
\leq& \frac{C_1}{n} \mathbb{E} \Big( \sum_{i \in \mathbb{Z}} b_2 \big( \big| v_{i+1} \big|^2
- \big| v_i \big|^2 \big) \Big) \leq \frac{2C}{n}\mathbb{E}(\|v(t)\|^2).
\end{align}Then from \eqref{z24} to \eqref{z25}, we have
\begin{small}\begin{align}
2 d_2 b_1 \mathbb{E} \left( \langle B v(t), B(\theta_n^2 v(t)) \rangle \right)
\geq -\frac{2C}{n}\mathbb{E}(\|v(t)\|^2).
\end{align}\end{small}Analogously, we can deduce that there exists $C > 0$ such that
\begin{align}
2 d_2 b_1 \mathbb{E} \left( \langle B u(t), B(\theta_n^2 u(t)) \rangle \right)
\geq -\frac{2C}{n}\mathbb{E}(\|u(t)\|^2).
\end{align}For  the first term on the right-hand side of \eqref{z23}-\eqref{b52}, by the Young inequality, we have
\begin{align}\label{26}
2 b_2 \mathbb{E} \left( \langle \theta_n f_1, \theta_n u(t) \rangle \right)
&\leq \mathbb{E} \left( b_2 \|\theta_n u(t)\|^2 \right)
+ \mathbb{E} \left( \sum_{|i| \geq n} b_2f_{1i}^2(t) \right),
\end{align}and
\begin{align}\label{27}
2 b_1 \mathbb{E} \left( \langle \theta_n f_2, \theta_n v(t) \rangle \right)
&\leq  \mathbb{E} \left( b_1 \|\theta_n v(t)\|^2 \right)
+  \mathbb{E} \left( \sum_{|i| \geq n} b_1f_{2i}^2(t) \right).
\end{align}Applying \eqref{b7} to diffusion terms, we infer that
\begin{align}
&b_2 \epsilon_1^2 \sum_{k=1}^\infty \mathbb{E} \left( \|\theta_n h_k + \theta_n \sigma_k(t,u(t))\|^2 \right)\notag\\
\leq& 2 b_2 \epsilon_1^2 \sum_{k=1}^\infty \mathbb{E} \left( \|\theta_n h_k\|^2 \right)
+ 2 b_2 \epsilon_1^2 \sum_{k=1}^\infty \mathbb{E} \left( \sum_{i \in \mathbb{Z}} \theta(i / n) \delta_{k,i}\tilde{\sigma}_{k}(t,u_i(t))^2 \right) \nonumber \\
\leq &2 b_2 \epsilon_1^2 \sum_{k=1}^\infty \mathbb{E} \left( \|\theta_n h_k\|^2 \right)
+ 4 \epsilon_1^2 b_2 \alpha^2 \sum_{k=1}^\infty \mathbb{E} \left( \sum_{i \in \mathbb{Z}} \theta^2(i / n) \delta_{k, i}^2 \right) +4 \epsilon_1^2 \alpha^2 \|\delta\|^2 \mathbb{E} \left( b_2 \|\theta_n u(t)\|^2 \right) \nonumber \\
\leq& 2 b_2 \epsilon_1^2 \sum_{k=1}^\infty \sum_{|i| \geq n} h_{k, i}^2
+ 4 b_2 \epsilon_1^2 \alpha^2 \sum_{|i| \geq n}\sum_{k=1}^{\infty} \delta_{k,i}^2
+ 4 b_2 \epsilon_1^2 \alpha^2  \mathbb{E} \left( b_2 \|\theta_n u(t)\|^2 \right),
\end{align}and
\begin{align}
&b_1 \gamma_1^2 \sum_{k=1}^\infty \mathbb{E} \left( \|\theta_n h_k + \theta_n \sigma_k(t,v(t))\|^2 \right)\notag\\
\leq& 2 b_1 \gamma_1^2 \sum_{k=1}^\infty \sum_{|i| \geq n} h_{k, i}^2
+ 4 b_1 \gamma_1^2 \alpha^2 \sum_{|i| \geq n}\sum_{k=1}^{\infty} \delta_{k,i}^2
+ 4 b_1 \gamma_1^2 \alpha^2  \mathbb{E} \left( b_2 \|\theta_n v(t)\|^2 \right).
\end{align}Likewise, we have
\begin{align}
&b_2\epsilon_2^2\sum_{k=1}^\infty \int_{|y_k|<1} \mathbb{E} \left[ \|\theta_n q_k(t,u(t), y_k)\|^2 + \theta_n \kappa_k(t)\|^2 \right] \nu_k(dy_k)\notag\\
\leq& 2 b_2\epsilon_2^2\sum_{k=1}^\infty \int_{|y_k|<1} \mathbb{E} \left[ \|\theta_n q_k(t,u(t), y_k)\|^2 \right] \nu_k(dy_k)
+ 2 b_2\epsilon_2^2\sum_{k=1}^\infty \|\theta_n \kappa_k(t)\|^2 \nonumber \\
= &2 b_2\epsilon_2^2\sum_{i \in \mathbb{Z}} \sum_{k=1}^\infty \int_{|\nu_k|<1} \mathbb{E} \left[ \theta^2 \left( \frac{i}{n} \right) |q_{k, i}(t,u_i(t), y_k)|^2 \right] \nu_k(dy_k)
+ 2 b_2\epsilon_2^2\sum_{i \in \mathbb{Z}} \sum_{k=1}^\infty \theta^2 \left( \frac{i}{n} \right) |\kappa_{k, i}(t)|^2 \nonumber \\
\leq &4b_2\epsilon_2^2\alpha^2\sum_{i \in \mathbb{Z}} \sum_{k=1}^\infty \theta^2 \left( \frac{i}{n} \right) \delta_{k, i} + 4b_2\epsilon_2^2\alpha^2 \|\delta\|^2\sum_{i \in \mathbb{Z}} \sum_{k=1}^\infty \mathbb{E} \left[ \theta^2 \left( \frac{i}{n} \right) |u_i(t)|^2 \right]
+ 2b_2\epsilon_2^2 \sum_{|i|\geq n } \sum_{k=1}^\infty |\kappa_{k, i}(t)|^2 \nonumber \\
\leq& 4 b_2\epsilon_2^2\alpha^2\sum_{|i| \geq n} \sum_{k=1}^\infty \delta_{k, i} + 2 b_2\epsilon_2^2\sum_{|i| \geq n} \sum_{k=1}^\infty |\kappa_{k, i}(t)|^2
+ 4 b_2\epsilon_2^2\alpha^2\|\delta\|^2 \mathbb{E} \left[ \|\theta_n u(t)\|^2 \right],
\end{align}and
\begin{align}\label{z40}
&b_1\gamma_2^2\sum_{k=1}^\infty \int_{|y_k|<1} \mathbb{E} \left[ \|\theta_n q_k(t,v(t), y_k)\|^2 + \theta_n \kappa_k(t)\|^2 \right] \nu_k(dy_k)\notag\\
\leq& 4 b_1\gamma_2^2\alpha^2\sum_{|i| \geq n} \sum_{k=1}^\infty \delta_{k, i} + 2b_1\gamma_2^2 \sum_{|i| \geq n} \sum_{k=1}^\infty |\kappa_{k, i}(t)|^2
+ 4 b_1\gamma_2^2\alpha^2\|\delta\|^2 \mathbb{E} \left[ \|\theta_n v(t)\|^2 \right].
\end{align}Combining the preceding relations and inequalities \eqref{z23}-\eqref{z40}, we may conclude that
\begin{align}\label{b61}
&D^+\mathbb{E}\Big[b_1\|\theta_nv(s)\|^2+b_2\|\theta_nu(s)\|^2\Big]\notag\\
\leq&-2\varpi\mathbb{E}\Big(b_2\|\theta_nu(s)\|^2+b_1\|\theta_nv(s)\|^2\Big)+
\Bigl(4b_2\epsilon_1^2+4b_1\gamma_1^2+4b_2\epsilon_2^2+4b_1\gamma_2^2\Bigl)\alpha^2\sum_{|i|\geq n}\sum_{k=1}^\infty\delta_{k,i}^2(s)\notag\\
&+M_1\mathbb{E}\Big[\sum_{|i|\geq n}\sum_{k=1}^\infty|\kappa_{k,i}(s)|^2+\sum_{|i|\geq n}\sum_{k=1}^\infty|h_{k,i}(s)|^2+b_1\sum_{|i|\geq n}|f_{2,i}(s)|^2+b_2\sum_{|i|\geq n}|f_{1,i}(s)|^2\Big]\notag\\
&+\frac{4C}{n}\mathbb{E}\left(b_2\|u(t)\|^2+b_1\|v(t)\|^2\right),
\end{align}where $M_1=(1+2b_2\epsilon_1^2+
2b_1\gamma_1^2+2b_1\gamma_2^2+2b_2\epsilon_2^2)$, and $D^+$ is the upper right Dini derivative. Given $t \in \mathbb{R}^+$ and $\tau \in \mathbb{R}$, integrating the above over $(\tau - t, \tau)$, we obtain
\begin{align}\label{TTT1}
&\mathbb{E} \left( b_2\|\theta_nu(\tau, \tau - t, u_\tau)\|^2+b_1\|\theta_nv(\tau, \tau - t, v_\tau)\|^2 \right)\notag\\
\leq&  C \int_{-\infty}^\tau e^{-\varpi(\tau - s)} \left( \sum_{|i|\geq n}\sum_{k=1}^\infty|\kappa_{k,i}(s)|^2+\sum_{|i|\geq n}\sum_{k=1}^\infty|h_{k,i}(s)|^2+\sum_{|i|\geq n}\left(b_1|f_{2,i}(s)|^2+b_2|f_{1,i}(s)|^2\right) \right) \mathrm{d}s\notag\\
&+\mathbb{E} \left( \|\theta_n\xi\|^2_{\ell^2\times\ell^2} \right) e^{-\varpi t}+C\int_{\tau-t}^\tau\sum_{|i|\geq n}\sum_{k=1}^\infty|\delta_{k,i}(s)|^2\,\mathrm{d}s\notag\\
&+\frac{4C}{n}\int_{\tau - t}^\tau e^{-\varpi(\tau - s)} \mathbb{E} \left( b_2\|u(s, \tau - t, u_\tau)\|^2+b_1\|v(s, \tau - t, v_\tau)\|^2 \right) \mathrm{d}s
\end{align}
By Lemma \ref{yingli41} we have
\begin{align}\label{TTT2}
\frac{4C}{n} \int_{\tau - t}^\tau e^{-\varpi(\tau - s)} \mathbb{E} \left( b_2\|u(s, \tau - t, u_\tau)\|^2+b_1\|v(s, \tau - t, v_\tau)\|^2 \right) \mathrm{d}s \leq Ce^{-\varpi t}\mathbb{E} \left(\|\xi\|^2_{\ell^2\times\ell^2} \right)+ \frac{C}{n}R(\tau),
\end{align}
where $R(\tau)$ is given by \eqref{op}.
Then for every $\eta > 0$ and $\tau\in\mathbb{R}$,   there exists   $N_1 = N_1( \tau, \eta) \in \mathbb{N}$ such that for all $n \geq N_1$,
$$\frac{C}{n}R(\tau)<\frac{\eta}{2}.$$
 By \eqref{bc0}, we know for every $\eta > 0$ and $\tau \in \mathbb{R}$, there exists $N_2 = N_2(\tau, \eta) \in \mathbb{N}$ such that for $n > N_2$,
\begin{small}\begin{align}\label{TTT3}
&C\int_{ -\infty}^\tau e^{- \varpi(\tau - s)} \left( \sum_{|i|\geq n}\sum_{k=1}^\infty|\kappa_{k,i}(s)|^2+\sum_{|i|\geq n}\sum_{k=1}^\infty|h_{k,i}(s)|^2\right) \mathrm{d}s+C\int_{\tau-t}^\tau e^{-\varpi(\tau - s)} \sum_{|i|\geq n}\sum_{k=1}^\infty|\delta_{k,i}(s)|^2\,\mathrm{d}s\notag\\
&+C\int_{\tau - t}^\tau e^{-\varpi(\tau - s)} \left(\sum_{|i|\geq n}|f_{2,i}(s)|^2+\sum_{|i|\geq n}|f_{1,i}(s)|^2 \right)\,\mathrm{d}s\notag\\
\leq&\frac{\eta}{2}.
\end{align}\end{small}Combining \eqref{TTT1}-\eqref{TTT3}, we get for every $\eta > 0$  and  $\tau \in \mathbb{R}$, there exists  $N = \max(N_1, N_2)$ such that for all $0 < \epsilon_{i},\gamma_{i} \leq 1$, $i=1,2$,  and $n \geq N$,
\begin{align}
\mathbb{E} \left( \|\theta_n \varphi(t)\|^2 \right)  \leq   Ce^{-\varpi t}\mathbb{E} \left(\|\xi\|^2_{\ell^2\times\ell^2} \right)+ \eta.
\end{align} This completes the proof.
\end{proof}

\section{Existence of pullback measure attractors}
 This section is devoted to the existence, uniqueness and periodicity of $\mathcal{D}$-pullback measure attractors of \eqref{b8}-\eqref{bc1} in $\mathcal{P}_2(\ell^2\times\ell^2)$.
As usual, if $\Im : \ell^2\times\ell^2 \to \mathbb{R}$ is a bounded Borel function, then for $r \leq t$ and $\xi \in \ell^2\times\ell^2$, we set
\begin{align}
p(t, r)\Im(\xi) = \mathbb{E} \left( \Im(\varphi(t, r, \xi)) \right),
\end{align}
and
\begin{align}
p(r,\xi;t,\Gamma ) = (p(t, r)1_\Gamma)(\xi),
\end{align}
where $\Gamma \in \mathcal{B}(l^2)$ and $1_\Gamma$ is the characteristic function of $\Gamma$.

The following properties of $p(t, r)$, $r \leq t$, are standard  and the proof is omitted.

\begin{lemma}\label{UU1}
Suppose \eqref{b6}-\eqref{bc0} and \eqref{xs} hold. Then:
\begin{enumerate}[(i)]
    \item The family $\{p(t, r)\}_{r \leq t}$ is Feller; that is, for any $r \leq t$, the function $p(t, r)\phi$ $\in C_b(\ell^2\times\ell^2)$ if $\phi\in C_b(\ell^2\times\ell^2)$.
    \item For every $r \in \mathbb{R}$ and $\xi \in \ell^2\times\ell^2$, the process $\{\varphi(t, r, \xi)\}_{t \geq r}$ is an $\ell^2\times\ell^2$-valued Markov process.
\end{enumerate}
\end{lemma}

We will also investigate the periodicity of pullback measure attractors of system \eqref{b8}-\eqref{bc1} for which we assume that all given time-dependent functions are $\alpha$-periodic in $t \in \mathbb{R}$, for some $\chi > 0$, that is, for all $t \in \mathbb{R}$ and $k \in \mathbb{N}$,
\begin{align}\label{xs1}
&\kappa_k(t+\chi)=\kappa_k,\quad q_k(t+\chi,\cdot,y_k)=q_k(t,\cdot,y_k), \nonumber \\
&f_1(t + \chi) = f_1(t),\quad f_2(t + \chi) = f_2(t),\nonumber\\
&h_k(t + \chi) = h_k(t), \quad \delta_k(t + \chi,\cdot) = \delta_k(t,\cdot).
\end{align}By the similar argument as that of Lemma 4.1 in \cite{mnbv}, we get the following lemma.

\begin{lemma}\label{z43}
Suppose \eqref{b6}-\eqref{bc0}, \eqref{xs} and \eqref{xs1} hold. Then we have the family $\{p(t, r)\}_{r \leq t}$ is $\chi$-periodic; that is, for all $t \geq r$,
\begin{align*}
p(r, \xi;t,\cdot)= p(r + \chi, \xi;t + \chi,\cdot), \quad \forall \xi \in \ell^2\times\ell^2.
\end{align*}
\end{lemma}

Given $t \geq r$ and $\mu \in \mathcal{P}(\ell^2\times\ell^2)$, define
\begin{align}\label{z44}
p^*(t, r)\mu(\cdot) = \int_{\ell^2\times\ell^2} p(r, \xi;t,\cdot)\mu(d\xi).
\end{align}

Then $p^*(t, r) : \mathcal{P}(\ell^2\times\ell^2) \to \mathcal{P}(\ell^2\times\ell^2)$ is the dual operator of $p(t, r)$. For all $t \geq r$, $p^*(t, r)$ maps $\mathcal{P}_2(\ell^2\times\ell^2)$ to $ \mathcal{P}_2(\ell^2\times\ell^2)$.

We now define a non-autonomous dynamical system $S(t, \tau) : t \geq \tau$, for the family of operators $p^*(t, \tau)$. Given $t\in\mathbb{R}^+$ and $\tau\in\mathbb{R}$, let $S(t,\tau) : \mathcal{P}_2(\ell^2\times\ell^2) \to \mathcal{P}_2(\ell^2\times\ell^2)$ be the map given by
\begin{align}
S(t, \tau)\mu = p^*(\tau+t,\tau)\mu, \quad \mu \in \mathcal{P}_2(\ell^2\times\ell^2).
\end{align}
Note that the mapping
\(
S(t,\tau):\mathcal{P}_2(\ell^2\times\ell^2)\to\mathcal{P}_2(\ell^2\times\ell^2)
\)
can be equivalently characterized by the following duality relation:
\begin{equation}\label{eq:Phi-dual}
  \int_{\ell^2\times\ell^2}\!\phi(x)\,[S(t,\tau)\mu](dx)
  =\int_{\ell^2\times\ell^2}\!\mathbb{E}\!\left[\phi\!\big(\varphi(\tau+t,\tau,x)\big)\right]\mu(dx),
  \quad \forall\,\mu\in\mathcal{P}_2(\ell^2\times\ell^2),\ \psi\in C_b(\ell^2\times\ell^2).
\end{equation}

\begin{lemma}\label{h16}
Suppose \eqref{b6}-\eqref{bc0}, \eqref{xs} and \eqref{xs1} hold. Then $S(t, \tau)$, $t \geq \tau$, is a continuous non-autonomous dynamical system in $\mathcal{P}_2(\ell^2\times\ell^2)$ generated by \eqref{b8}-\eqref{bc1}. That is, $S(t, r) : \mathcal{P}_2(\ell^2\times\ell^2) \to \mathcal{P}_2(\ell^2\times\ell^2)$ satisfies the following conditions:
\begin{enumerate}[(i)]
    \item $S(0, \tau) = I_{\mathcal{P}_2(\ell^2\times\ell^2)}$, for all $\tau \in \mathbb{R}$;
    \item $S(s + t, \tau) = S(t, \tau + s) \circ S(s, \tau)$, for any $\tau \in \mathbb{R}$ and $t, s \in \mathbb{R}^+$,
    \item $S(t, \tau) : \mathcal{P}_2(\ell^2\times\ell^2) \to \mathcal{P}_2(\ell^2\times\ell^2)$ is continuous, for every $\tau \in \mathbb{R}$ and $t \in \mathbb{R}^+$.
        \item $S(t, \tau)\mathcal{P}_2(\ell^2\times\ell^2)\subseteq\mathcal{P}_2(\ell^2\times\ell^2)$ for all $t\geq0, \tau\in\mathbb{R}$.
\end{enumerate}
\end{lemma}

\begin{proof}
Note that (i) follows from the definition of $S$, and (ii) follows the Markov property of the solutions of \eqref{b8}-\eqref{bc1}.

We now prove (iii). Suppose $\mu_n \to \mu$ in $\mathcal{P}_2(\ell^2\times\ell^2)$. We will show $S(t, \tau)\mu_n \to S(t, \tau)\mu$ in $\mathcal{P}_2(\ell^2\times\ell^2)$ for every $\tau \in \mathbb{R}$ and $t \in \mathbb{R}^+$. Let $\phi \in C_b(\ell^2\times\ell^2)$. By Lemma \ref{UU1}, we have $p(\tau + t, \tau)\phi \in C_b(\ell^2\times\ell^2)$ for all $\tau \in \mathbb{R}$ and $t \in \mathbb{R}^+$, and hence
\begin{align*}
&\lim_{n \to \infty} (\phi, S(t, \tau)\mu_n)\\
 =& \lim_{n \to \infty} (\phi, p^*(t + \tau, \tau)\mu_n)\\
=&\lim_{n \to \infty} (p(t + \tau, \tau)\phi, \mu_n) = (p(t + \tau, \tau)\phi, \mu) \nonumber \\
=& (\phi, p^*(t + \tau, \tau)\mu) = (\phi, S(t, \tau)\mu),
\end{align*}
as desired.  We finally prove (iv).

Apply Levi's theorem for the functions $\phi_n(x)=\|x\|^2_{\ell^2\times\ell^2}\wedge n$, we can deduce from \eqref{eq:Phi-dual} that
\begin{align}\label{Phi-dual1}
&\int_{\ell^2\times\ell^2}\!\|x\|^2_{\ell^2\times\ell^2}\,[S(t,\tau)\mu](dx)
 =\lim_{n\rightarrow\infty}
 \int_{\ell^2\times\ell^2}\!\phi_n(x)\,[S(t,\tau)\mu](dx)\notag\\
&=\lim_{n\rightarrow\infty}
  \int_{\ell^2\times\ell^2}\!\mathbb{E}[\phi_n(\varphi(t+\tau,\tau,x))]\mu(dx)
  =\int_{\ell^2\times\ell^2}\!\mathbb{E}\!\left[\|\varphi(\tau+t,\tau,x)\|^2_{\ell^2\times\ell^2}
  \right]\mu(dx).
\end{align}
Then by Lemma \ref{yingli41} we get for all $\mu\in\mathcal{P}_2(\ell^2\times\ell^2)$,
\begin{align}\label{Phi-dual2}
\|S(t,\tau)\mu\|^2_{\mathcal{P}_2(\ell^2\times\ell^2)}&=\int_{\ell^2\times\ell^2}
\!\|x\|^2_{\ell^2\times\ell^2}\,[S(t,\tau)\mu](dx)\notag\\
  &=\int_{\ell^2\times\ell^2}\!\mathbb{E}\!\left[\|\varphi(\tau+t,\tau,x)\|^2_{\ell^2\times\ell^2}
  \right]\mu(dx)\notag\\
  &\leq  C e^{-\varpi t} \int_{\ell^2\times\ell^2}\|x\|^2_{\ell^2\times\ell^2}\mu(dx)\notag\\
  &\leq  C\|\mu\|^2_{\mathcal{P}_2(\ell^2\times\ell^2)}   +CR(\tau+t)<\infty.
\end{align}
Thus, $S(t, \tau)\mu\in\mathcal{P}_2(\ell^2\times\ell^2).$

\end{proof}
By Lemma \ref{yingli41}, we obtain a $\mathcal{D}$-pullback absorbing set for $S$ as stated below.

\begin{lemma}\label{h17}
Suppose that \eqref{b6}--\eqref{bc0}, \eqref{xs}, and \eqref{xs1} hold.
Given $\tau\in\mathbb{R}$, define
\begin{align}\label{h11}
K(\tau) = B_{\mathcal{P}_2(\ell^2\times\ell^2)}\!\left(L_1^{1/2}(\tau)\right),
\end{align}
where
\begin{align}\label{h12}
L_1(\tau)
&= C\!\int_{-\infty}^\tau e^{-\varpi(\tau - s)}
   \Bigg(\sum_{k=1}^\infty\|\kappa_k(s)\|^2 + \sum_{k=1}^\infty\|h(s)\|^2\Bigg)\mathrm{d}s
   + C\!\int_{-\infty}^\tau e^{-\varpi(\tau - s)}\|\delta(s)\|^2\,\mathrm{d}s \notag\\
&\quad + C\!\int_{-\infty}^\tau e^{-\varpi(\tau - s)}\!\left(\|f_2(s)\|^2+\|f_1(s)\|^2\right)\mathrm{d}s.
\end{align}
Then the family $K=\{K(\tau):\tau\in\mathbb{R}\}$ belongs to $\mathcal{D}$ and constitutes a closed $\mathcal{D}$-pullback absorbing set for the cocycle $S$.
\end{lemma}

\begin{proof}
The argument follows the strategy of \cite{Lime1,Lime2}.

\emph{Step 1: Closedness of $K(\tau)$.}
Let $\{\mu_n\}_{n\in\mathbb{N}}\subset K(\tau)$ and assume $\mu_n\rightharpoonup\mu_0$ in $\mathcal{P}(\ell^2\times\ell^2)$.
By the definition of $K(\tau)$, we have
\begin{equation}\label{eq:uniform-moment}
\sup_{n\in\mathbb{N}}\int_{\ell^2\times\ell^2}\|x\|_{\ell^2\times\ell^2}^{2}\,\mu_n(dx)\le L_1(\tau).
\end{equation}
For each $m\in\mathbb{N}$, define the bounded continuous truncation
\[
\phi_m(x):=\|x\|_{\ell^2\times\ell^2}^{2}\wedge m \in C_b(\ell^2\times\ell^2),
\quad 0\le\phi_m\uparrow \|x\|_{\ell^2\times\ell^2}^{2}.
\]
Since $\phi_m\in C_b(\ell^2\times\ell^2)$ and $\mu_n\rightharpoonup\mu_0$, it follows that, for each fixed $m$,
\[
\int_{\ell^2\times\ell^2}\phi_m(x)\,\mu_0(dx)
 = \lim_{n\to\infty}\int_{\ell^2\times\ell^2}\phi_m(x)\,\mu_n(dx)
 \le \sup_{n\in\mathbb{N}}\int_{\ell^2\times\ell^2}\|x\|_{\ell^2\times\ell^2}^{2}\,\mu_n(dx)
 \le L_1(\tau).
\]
Letting $m\to\infty$ and using the Beppo--Levi monotone convergence theorem yields
\[
\int_{\ell^2\times\ell^2}\|x\|_{\ell^2\times\ell^2}^{2}\,\mu_0(dx)
 = \lim_{m\to\infty}\int_{\ell^2\times\ell^2}\phi_m(x)\,\mu_0(dx)
 \le L_1(\tau).
\]
Hence $\mu_0\in K(\tau)$, and therefore each $K(\tau)$ is closed.

\emph{Step 2: Pullback absorption property.}
Fix $\tau\in\mathbb{R}$ and $D\in\mathcal{D}$.
For any $t\ge0$ and $\mu\in D(\tau-t)$, by \eqref{eq:Phi-dual} and Lemma~\ref{yingli41}, we obtain
\begin{align}\label{closed3}
\|S(t,\tau-t)\mu\|_{\mathcal{P}_2(\ell^2\times\ell^2)}^{2}
 &= \int_{\ell^2\times\ell^2}\|x\|_{\ell^2\times\ell^2}^{2}[S(t,\tau-t)\mu](dx) \notag\\
 &= \int_{\ell^2\times\ell^2}\mathbb{E}\!\left[\|\varphi(\tau,\tau-t,x)\|_{\ell^2\times\ell^2}^{2}\right]\mu(dx) \notag\\
 &\le C e^{-\varpi t}\!\int_{\ell^2\times\ell^2}\|x\|_{\ell^2\times\ell^2}^{2}\mu(dx) + C R(\tau) \notag\\
 &\le C e^{-\varpi t}\|D(\tau-t)\|_{\mathcal{P}_2(\ell^2\times\ell^2)}^{2} + C R(\tau).
\end{align}
By the temperedness condition \eqref{temped}, there exists $T=T(\tau,D)>0$, independent of $\epsilon_i,\gamma_i$ $(i=1,2)$, such that for all $t\ge T$ and $0<\epsilon_i,\gamma_i\le1$,
\[
C e^{-\varpi t}\|D(\tau-t)\|_{\mathcal{P}_2(\ell^2\times\ell^2)}^{2}\le R(\tau).
\]
Combining this with \eqref{closed3} gives
\[
\|S(t,\tau-t)\mu\|_{\mathcal{P}_2(\ell^2\times\ell^2)}^{2}\le L_1(\tau),
\quad \forall\, t\ge T,
\]
which implies \(S(t,\tau-t)D(\tau-t)\subseteq K(\tau)\).
Hence, \(K\) is a pullback absorbing set of \(S\) in \(\mathcal{P}_2(\ell^2\times\ell^2)\).

\emph{Step 3: Membership of \(K\) in \(\mathcal{D}\).}
By \eqref{h11}, \eqref{h12}, and \eqref{bc0}, we have
\[
e^{\varpi\tau}\|K(\tau)\|_{\mathcal{P}_2(\ell^2\times\ell^2)}^{2}
= e^{\varpi\tau}L_1(\tau) \to 0, \quad \text{as } \tau\to -\infty.
\]
Therefore, \(K=\{K(\tau):\tau\in\mathbb{R}\}\in\mathcal{D}\).
This completes the proof.
\end{proof}

We now present the $\mathcal{D}$-pullback asymptotically compact of $S$ associated with \eqref{b8}-\eqref{bc1}.

\begin{lemma}\label{h18}
Suppose that \eqref{b6}--\eqref{bc0} and \eqref{xs} hold.
Then $S$ is $\mathcal{D}$-pullback asymptotically compact in $\mathcal{P}_2(\ell^2\times\ell^2)$;
that is, for every $\tau\in\mathbb{R}$, whenever $t_n\to+\infty$ and $\mu_n\in D(\tau-t_n)$ with $D\in\mathcal{D}$,
the sequence $\{S(t_n,\tau-t_n)\mu_n\}_{n\in\mathbb{N}}$ has a convergent subsequence in $\mathcal{P}_2(\ell^2\times\ell^2)$.
\end{lemma}

\begin{proof}
By the pullback absorption property established in Lemma~\ref{h17} and the duality relation \eqref{eq:Phi-dual},
there exist $N_1=N_1(\tau,D)\in\mathbb{N}$ and a constant $M=M(\tau)>0$, independent of $\epsilon$ and $D$,
such that for all $n>N_1$,
\begin{align}\label{h24}
\|S(t_n,\tau-t_n)\mu_n\|_{\mathcal{P}_2(\ell^2\times\ell^2)}^{2}
&= \int_{\ell^2\times\ell^2}\!\mathbb{E}\!\left[\|\varphi(\tau,\tau-t_n,x)\|_{\ell^2\times\ell^2}^{2}\right]\mu_n(dx)
\le M.
\end{align}
By Chebyshev's (Markov's) inequality, \eqref{h24} implies that for any $R_1>0$ and all $n>N_1$,
\[
\int_{\ell^2\times\ell^2}\mathbb{P}\!\left(\|\varphi(\tau,\tau-t_n,x)\|_{\ell^2\times\ell^2}^{2}>R_1\right)\mu_n(dx)
\le \frac{M}{R_1}\xrightarrow[R_1\to\infty]{}0.
\]
Hence, for every $\tau\in\mathbb{R}$, $\eta>0$, and $m\in\mathbb{N}$, one can choose $R_2=R_2(\tau,\eta,m)>0$ such that,
for all $n>N_1$,
\begin{equation}\label{h14}
\int_{\ell^2\times\ell^2}\mathbb{P}\!\left(\|\varphi(\tau,\tau-t_n,x)\|_{\ell^2\times\ell^2}^{2}>R_2\right)\mu_n(dx)
< \frac{\eta}{2^{m+1}}.
\end{equation}

Since $\mu_n\in D(\tau-t_n)$ and $D\in\mathcal{D}$,
the temperedness condition \eqref{temped} ensures that there exists $H_m=H_m(\tau,D,\eta,m)>N_1$ such that
for all $n\ge H_m$,
\[
C e^{-\varpi t_n}\|\mu_n\|_{\mathcal{P}_2(\ell^2\times\ell^2)}^{2}
\le C e^{-\varpi t_n}\|D(\tau-t_n)\|_{\mathcal{P}_2(\ell^2\times\ell^2)}^{2}
< \frac{\eta}{2^{2m+2}}.
\]
Letting $\xi(\omega)\equiv x\in\ell^2\times\ell^2$ in \eqref{weba1},
there exists an integer $n_m=n_m(\tau,D,\eta,m)$ such that for all $n\in\mathbb{N}$,
\[
\int_{\ell^2\times\ell^2}\mathbb{E}\!\left[\sum_{|i|>n_m}|\varphi_i(\tau,\tau-t_n,x)|^2\right]\mu_n(dx)
< \frac{\eta}{2^{2m+2}}
   + C e^{-\varpi t_n}\!\int_{\ell^2\times\ell^2}\|x\|_{\ell^2\times\ell^2}^{2}\mu_n(dx).
\]
Combining these two estimates yields, for all $n\ge H_m$,
\[
\int_{\ell^2\times\ell^2}\mathbb{E}\!\left[\sum_{|i|>n_m}|\varphi_i(\tau,\tau-t_n,x)|^2\right]\mu_n(dx)
< \frac{\eta}{2^{2m+1}}.
\]
Applying Chebyshev's inequality again gives
\begin{align}\label{h15}
\int_{\ell^2\times\ell^2}\mathbb{P}\!\left(
\sum_{|i|>n_m}|\varphi_i(\tau,\tau-t_n,x)|^2>\frac{1}{2^{m}}
\right)\mu_n(dx)
\le 2^{m}\!\int_{\ell^2\times\ell^2}\!\mathbb{E}\!\left[\sum_{|i|>n_m}|\varphi_i|^2\right]\mu_n(dx)
< \frac{\eta}{2^{m+1}}.
\end{align}

Define
\begin{align}\label{z41}
Y_{1,m} &= \Big\{\varphi\in\ell^2\times\ell^2:\ \|\varphi\|_{\ell^2\times\ell^2}^{2}\le R_2\Big\},\\
Y_{2,m} &= \Big\{\varphi\in\ell^2\times\ell^2:\ \sum_{|i|>n_m}(|u_i|^2+|v_i|^2)\le 2^{-m}\Big\},
\end{align}
and let \(Y_m:=Y_{1,m}\cap Y_{2,m}\).
From \eqref{z41}, the set
\(\{(\varphi_i)_{|i|\le n_m}:\varphi\in Y_m\}\)
is bounded in the finite-dimensional space
\(\mathbb{R}^{2n_m+1}\times\mathbb{R}^{2n_m+1}\),
and hence precompact.
Therefore, it can be covered by finitely many open balls of radius \(2^{-m/2}\).
Together with the small tail estimate from \(Y_{2,m}\),
this implies that \(Y_m\) can be covered by finitely many open balls of radius \(2^{-(m-1)/2}\)
in \(\ell^2\times\ell^2\).

For each fixed $\tau\in\mathbb{R}$ and $m\in\mathbb{N}$,
we can choose a compact set $K_m=K_m(\tau)$ such that for all $n\ge H_m$,
\[
\int_{\ell^2\times\ell^2}\mathbb{P}\!\left(\varphi(\tau,\tau-t_n,x)\in K_m\right)\mu_n(dx)
> 1-\frac{\eta}{2^{m}}.
\]
Combining this with \eqref{h14}--\eqref{h15}, define
\(\mathcal{Y}_m:=Y_m\cup K_m\).
Then $\mathcal{Y}_m$ admits a finite open cover of balls of radius \(2^{-(m-1)/2}\) in \(\ell^2\times\ell^2\),
and for all \(n\in\mathbb{N}\),
\[
\int_{\ell^2\times\ell^2}\mathbb{P}\!\left(\varphi(\tau,\tau-t_n,x)\in\mathcal{Y}_m\right)\mu_n(dx)
> 1-\frac{\eta}{2^{m}}.
\]
Let \(\mathcal{Y}:=\bigcap_{m=1}^\infty\mathcal{Y}_m\).
Then $\mathcal{Y}$ is closed and totally bounded, hence compact in $\ell^2\times\ell^2$.
Moreover, for every $n\in\mathbb{N}$,
\[
\int_{\ell^2\times\ell^2}\mathbb{P}\!\left(\varphi(\tau,\tau-t_n,x)\in\mathcal{Y}\right)\mu_n(dx)
\ge 1-\sum_{m=1}^{\infty}\frac{\eta}{2^m}=1-\eta.
\]
Since $\eta>0$ is arbitrary, the family
\(\{S(t_n,\tau-t_n)\mu_n\}\) is tight in $\mathcal{P}(\ell^2\times\ell^2)$, in combination with Prohorov's theorem, we completes the proof.
\end{proof}

Next, we establish the existence, uniqueness, and periodicity of $\mathcal{D}$-pullback measure attractors for \eqref{b8}-\eqref{bc1} on $\mathcal{P}_2(\ell^2\times\ell^2)$.

\begin{theorem}\label{QQQ1}
Suppose \eqref{b6}-\eqref{bc0}, and \eqref{xs} hold. Then for every $0 < \epsilon_i,\gamma_i \leq 1$, $i=1,2$, the system $S$ associated with \eqref{b8}-\eqref{bc1} has a unique $\mathcal{D}$-pullback measure attractor $\mathcal{A} = \{\mathcal{A}(\tau) : \tau \in \mathbb{R}\} \in \mathcal{P}_2(\ell^2\times\ell^2)$, which is given by, for each $\tau \in \mathbb{R}$,
\begin{align}
\mathcal{A}(\tau) = \omega(K, \tau) = \{\psi(0, \tau) : \psi \text{ is a } \mathcal{D}\text{-complete orbit of } S\} \nonumber \\
= \{\xi(\tau) : \xi \text{ is a } \mathcal{D}\text{-complete solution of } S\},
\end{align}
where $K = \{K(\tau) : \tau \in \mathbb{R}\}$ is the $\mathcal{D}$-pullback absorbing set of $S$ as given by Lemma \ref{h17}.
\end{theorem}

\begin{proof}
It follows from Lemma \ref{h16} that $S$ is a continuous non-autonomous dynamical system on $\mathcal{P}_2(\ell^2\times\ell^2)$. Notice that $S$ has a closed $\mathcal{D}$-pullback absorbing set $K$ in $\mathcal{P}_2(\ell^2\times\ell^2)$ by Lemma \ref{h17} and is $\mathcal{D}$-pullback asymptotically compact in $\mathcal{P}_2(\ell^2\times\ell^2)$ by Lemma \ref{h18}. Hence the existence and uniqueness of the $\mathcal{D}$-pullback measure attractor for $S$ follows from Proposition \ref{h19} immediately.
\end{proof}
We now consider the periodicity of the measure attractor $\mathcal{A}$.

\begin{theorem}
Suppose \eqref{b6}-\eqref{bc0}, \eqref{xs} and \eqref{xs1} hold. then for every $0 < \epsilon_i,\gamma_i \leq 1$, $i=1,2$, $S$ associated with \eqref{b8}-\eqref{bc1} has a unique $\chi$-periodic $\mathcal{D}$-pullback measure attractor $\mathcal{A}$ in $\mathcal{P}_2(\ell^2\times\ell^2)$.
\end{theorem}
\begin{proof}By \eqref{h11} and \eqref{h12}, we find that $K$ is $\omega$-periodic. In addition, it follows from Lemma \ref{z43} and \eqref{z44}, the non-autonomous dynamical system $S$ associated with system \eqref{b8}-\eqref{bc1} is also $\chi$-periodic. $\chi$-periodic. Thus, from proposition \ref{h19}, we can infer the periodicity of the measure attractor $\mathcal{A}$.
\end{proof}
\section{Upper semicontinuity of pullback measure attractors}

In this section, we prove the upper semicontinuity of $\mathcal{D}$-pullback measure attractors for the non-autonomous stochastic lattice systems as the noise intensity ($\epsilon_i,\gamma_i,i=1,2$)$\rightarrow0$ , \textbf{\textcolor{red}{denote by $\lambda_0=(0,0,0,0)$, $\lambda=(\epsilon_1,\epsilon_2,\gamma_1,\gamma_2)$.}}

We analyze the non-autonomous stochastic lattice systems given by equations \eqref{b8}-\eqref{bc1}. It is worth noting that all results derived in the preceding sections hold true for the case where $\lambda\rightarrow\lambda_0$. To emphasize the dependence of the solutions on the parameters $\epsilon_i$ and $\gamma_i$ for $i = 1, 2$, we introduce the notation $\varphi^{\lambda}(t, \tau, \xi)$ as the solution to the system \eqref{b8}-\eqref{bc1}, where $\tau$ represents the initial time and $\xi \in L_{\mathcal{F}_\tau}^2(\Omega, \ell^2 \times \ell^2)$ is the initial state.

With the parameters $\epsilon_i, \gamma_i \in [0, 1]$, let $p^{\lambda}(t, \tau)$ denote the transition operator corresponding to the solution $\varphi^{\lambda}(t, \tau, \xi)$, while $(p^{\lambda})^*(t, \tau)$ represents the dual operator associated with $p^{\lambda}(t, \tau)$. For each $t \in \mathbb{R}^+$ and $\tau \in \mathbb{R}$, we define the non-autonomous dynamical system  $S^{\lambda}(t, \tau): \mathcal{P}_2(\ell^2 \times \ell^2)\rightarrow\mathcal{P}_2(\ell^2 \times \ell^2)$ be the map given by
\begin{align}
S^{\lambda}(t, \tau) \mu = (p^{\lambda})^*(\tau + t, \tau) \mu, \quad \forall \mu \in \mathcal{P}_2(\ell^2 \times \ell^2).
\end{align}

Let $\mathcal{A}^{\lambda}$ be the $\mathcal{D}$-pullback measure attractor of $S^{\lambda}$, we establish the convergence of solutions of problem \eqref{b8}-\eqref{bc1} when $\epsilon_i,\gamma_i,i=1,2 \to 0$.

\begin{lemma}\label{h22}
Suppose \eqref{b6}-\eqref{b7} hold. Then  given $\tau \in \mathbb{R}$ and $T>0$, we have
\begin{align}\label{h22io}
\lim_{\lambda \to \lambda_0} \sup_{\mu \in \mathcal{A}(\tau)} d_{\mathcal{P}_2(\ell^2\times\ell^2)}\left(\widetilde{S}^{\lambda}(t, \tau)\mu, S^{\lambda_0}(t, \tau)\widetilde{\mu}\right) = 0.
\end{align}where $\widetilde{S}(\tau + T, \tau)\mu$ denotes the null expansion of $S_N(\tau + T, \tau)\mu$ and the metric $d_{\mathcal{M}(X)}$ is defined by
\begin{align*}
d_{\mathcal{P}(X)} (\mu_1, \mu_2) = \sup_{\phi \in L_b(X), \|\phi\|_L \leq 1} |(\phi, \mu_1) - (\phi, \mu_2)|, \quad \forall \mu_1, \mu_2 \in \mathcal{P}(X).
\end{align*}
\end{lemma}

\begin{proof}
Let  $\varphi^\lambda(t)=\varphi^{(\epsilon_1,\epsilon_2,\gamma_1,\gamma_2)}(t,\tau,\varphi_\tau)$ and  $\varphi^{\lambda_0}(t)=\varphi^{(\epsilon_{1,0},\epsilon_{2,0},\gamma_{1,0},\gamma_{2,0})}(t,\tau,\varphi_\tau)$ satisfy that for all $t\in [\tau,\tau+T], \lambda \in [0, 1]\times[0, 1]\times[0, 1]\times[0, 1]$. By \eqref{b8}-\eqref{bc1} and the It\^{o} formula, we obtain
\begin{align}\label{ggg1}
 &\|u^\lambda(t)-u^{\lambda_0}(t)\|^2\notag\\ =&\|\xi^\lambda-\xi^{\lambda_0}\|^2 -2d_1\int_\tau^t\|B(u^\lambda(s)-u^{\lambda_0}(s))\|\mathrm{d}s -2a_1\int_\tau^t\|u^\lambda(s)-u^{\lambda_0}(s)\|^2\mathrm{d}s \notag\\
&+b_1\int_\tau^t\langle F(u^\lambda,v^\lambda)-F(u^{\lambda_0},v^{\lambda_0}),u^\lambda-u^{\lambda_0}\rangle\mathrm{d}s
-b_2\int_\tau^t\langle G(u^\lambda)-G(u^{\lambda_0}),u^\lambda-u^{\lambda_0}\rangle\mathrm{d}s\notag\\
&+2(\epsilon_1-\epsilon_{1,0})\int_\tau^t\sum_{k\in\mathbb{N}}\left\langle \aleph_k(s,u^\lambda(s)),u^\lambda(s)-u^{\lambda_0}(s)\right\rangle\mathrm{d}W_k(s)\notag\\ &+2\epsilon_{1,0}\int_\tau^t\sum_{k\in\mathbb{N}} \left\langle\sigma_k(s,u^\lambda)-\sigma_k(s,u^{\lambda_0}), u^\lambda(s)-u^{\lambda_0}(s)\right\rangle\mathrm{d}W_k(s)\notag\\ &+\int_\tau^t\sum_{k\in\mathbb{N}}\|(\epsilon_1-\epsilon_{1,0})\aleph_k(s,u^\lambda(s)) +\epsilon_{1,0}(\sigma(s,u^\lambda)-\sigma(s,u^{\lambda_0}))\|^2\mathrm{d}s\notag\\ &+\int_\tau^t\sum_{k\in\mathbb{N}}\int_{0<y_k<1} \Big(\|u^\lambda(s-)-u^{\lambda_0}(s-)+(\epsilon_2-\epsilon_{2,0})\hat{\aleph}_k(s,u) +\epsilon_{2,0}\Xi\|^2\notag\\
&-\|u^\lambda(s-) -u^{\lambda_0}(s-)\|^2\Big)\tilde{L}_k(\mathrm{d}s,\mathrm{d}y_k)\notag\\ &+\int_\tau^t\sum_{k\in\mathbb{N}}\int_{0<y_k<1}\| (\lambda-\lambda_0)\hat{\aleph}+\lambda_0\Xi\|^2\nu_k\mathrm{d}y_k\mathrm{d}s,
\end{align}and
\begin{align}\label{ggg2}
 &\|v^\lambda(t)-v^{\lambda_0}(t)\|^2\notag\\ =
 &\|\xi^\lambda-\xi^{\lambda_0}\|^2 -2d_2\int_\tau^t\|B(v^\lambda(s)-v^{\lambda_0}(s))\|\mathrm{d}s -2a_2\int_\tau^t\|v^\lambda(s)-v^{\lambda_0}(s)\|^2\mathrm{d}s \notag\\
&-b_1\int_\tau^t\langle F(u^\lambda,v^\lambda)-F(u^{\lambda_0},v^{\lambda_0}),v^\lambda-v^{\lambda_0}\rangle\mathrm{d}s
+b_2\int_\tau^t\langle G(u^\lambda)-G(u^{\lambda_0}),v^\lambda-v^{\lambda_0}\rangle\notag\\ &+2(\gamma_1-\gamma_{1,0})\int_\tau^t\sum_{k\in\mathbb{N}}\left\langle \aleph_k(s,v^\lambda(s)),v^\lambda(s)-v^{\lambda_0}(s)\right\rangle\mathrm{d}W_k(s)\notag\\ &+2\gamma_{1,0}\int_\tau^t\sum_{k\in\mathbb{N}} \left\langle\sigma_k(s,v^\lambda)-\sigma_k(s,v^{\lambda_0}), v^\lambda(s)-v^{\lambda_0}(s)\right\rangle\mathrm{d}W_k(s)\notag\\ &+\int_\tau^t\sum_{k\in\mathbb{N}}\|(\gamma_1-\gamma_{1,0})\aleph_k(s,v^\lambda(s)) +\gamma_{1,0}(\sigma(s,v^\lambda)-\sigma(s,v^{\lambda_0}))\|^2\mathrm{d}s\notag\\ &+\int_\tau^t\sum_{k\in\mathbb{N}}\int_{0<y_k<1} \Big(\|v^\lambda(s-)-v^{\lambda_0}(s-)+(\gamma_2-\gamma_{2,0})\hat{\aleph}_k(s,v) +\gamma_{2,0}\Xi\|^2\notag\\
&-\|v^\lambda(s-) -v^{\lambda_0}(s-)\|^2\Big)\tilde{L}_k(\mathrm{d}s,\mathrm{d}y_k)\notag\\ &+\int_\tau^t\sum_{k\in\mathbb{N}}\int_{0<y_k<1}\| (\gamma_2-\gamma_{2,0})\hat{\aleph}_k(s,v^\lambda(s))+\gamma_{2,0}\Xi\|^2\nu_k\mathrm{d}y_k\mathrm{d}s,
\end{align}where
\begin{align}
&\aleph_k(s,\cdot)\triangleq h_k(s)+\sigma_k(s,\cdot),\notag\\
&\hat{\aleph}_k(s,\cdot)\triangleq \kappa_k(s)+q(s,\cdot,y_k),\notag\\
&\Xi\triangleq q(s,\cdot,y_k)-q(s,\cdot,y_k).
\end{align}It follows from \eqref{b2}-\eqref{b5} and \eqref{ggg1}-\eqref{ggg2} that
\begin{align}\label{ppzz1}
&\mathbb{E}\left(\sup_{r\in[\tau,t]}\|\varphi^\lambda(r)-\varphi^{\lambda_0}(r)\|^2\right)\notag\\
\leq&\mathbb{E}\left(\|\xi^\lambda-\xi^{\lambda_0}\|^2\right)
+C\int_\tau^t\mathbb{E}\left(\sup_{r\in[\tau,s]}\|\varphi^\lambda(r)-\varphi^{\lambda_0}(r)\|^2\right)\mathrm{d}s\notag\\
&+2(\epsilon_1-\epsilon_{1,0})\mathbb{E}\left(\sup_{r\in[\tau,t]}\left|\int_\tau^r\sum_{k\in\mathbb{N}}\left\langle \aleph_k(s,u^\lambda(s)),u^\lambda(s)-u^{\lambda_0}(s)\right\rangle\mathrm{d}W_k(s)\right|\right)\notag\\ &+2\epsilon_{1,0}\mathbb{E}\left(\sup_{r\in[\tau,t]}\left|\int_\tau^r\sum_{k\in\mathbb{N}} \left\langle\sigma_k(s,u^\lambda)-\sigma_k(s,u^{\lambda_0}), u^\lambda(s)-u^{\lambda_0}(s)\right\rangle\mathrm{d}W_k(s)\right|\right)\notag\\ &+\mathbb{E}\left(\int_\tau^t\sum_{k\in\mathbb{N}}\|(\epsilon_1-\epsilon_{1,0})\aleph_k(s,u^\lambda(s)) +\epsilon_{1,0}(\sigma(s,u^\lambda)-\sigma(s,u^{\lambda_0}))\|^2\mathrm{d}s\right)\notag\\ &+\mathbb{E}\Big(\sup_{r\in[\tau,t]}\Big|\int_\tau^r\sum_{k\in\mathbb{N}}\int_{0<y_k<1} \Big(\|u^\lambda(s-)-u^{\lambda_0}(s-)+(\epsilon_2-\epsilon_{2,0})\hat{\aleph}_k(s,u) +\epsilon_{2,0}\Xi\|^2\notag\\
&-\|u^\lambda(s-) -u^{\lambda_0}(s-)\|^2\Big)\tilde{L}_k(\mathrm{d}s,\mathrm{d}y_k)\Big|\Big)\notag\\ &+\mathbb{E}\left(\int_\tau^t\sum_{k\in\mathbb{N}}\int_{0<y_k<1}\| (\epsilon_2-\epsilon_{2,0})\hat{\aleph}+\epsilon_{2,0}\Xi\|^2\nu_k\mathrm{d}y_k\mathrm{d}s\right)\notag\\
&+2(\gamma_1-\gamma_{1,0})\mathbb{E}\left(\sup_{r\in[\tau,t]}\left|\int_\tau^r\sum_{k\in\mathbb{N}}\left\langle \aleph_k(s,v^\lambda(s)),v^\lambda(s)-v^{\lambda_0}(s)\right\rangle\mathrm{d}W_k(s)\right|\right)\notag\\ &+2\gamma_{1,0}\mathbb{E}\left(\sup_{r\in[\tau,t]}\left|\int_\tau^r\sum_{k\in\mathbb{N}} \left\langle\sigma_k(s,v^\lambda)-\sigma_k(s,v^{\lambda_0}), v^\lambda(s)-v^{\lambda_0}(s)\right\rangle\mathrm{d}W_k(s)\right|\right)\notag\\ &+\mathbb{E}\left(\int_\tau^t\sum_{k\in\mathbb{N}}\|(\gamma_1-\gamma_{1,0})\aleph_k(s,v^\lambda(s)) +\gamma_{1,0}(\sigma(s,v^\lambda)-\sigma(s,v^{\lambda_0}))\|^2\mathrm{d}s\right)\notag\\ &+\mathbb{E}\Big(\sup_{r\in[\tau,t]}\Big|\int_\tau^r\sum_{k\in\mathbb{N}}\int_{0<y_k<1} \Big(\|v^\lambda(s-)-v^{\lambda_0}(s-)+(\gamma_2-\gamma_{2,0})\hat{\aleph}_k(s,v) +\gamma_{2,0}\Xi\|^2\notag\\
&-\|v^\lambda(s-) -v^{\lambda_0}(s-)\|^2\Big)\tilde{L}_k(\mathrm{d}s,\mathrm{d}y_k)\Big|\Big)\notag\\ &+\mathbb{E}\left(\int_\tau^t\sum_{k\in\mathbb{N}}\int_{0<y_k<1}\| (\gamma_2-\gamma_{2,0})\hat{\aleph}_k(s,v^\lambda(s))
+\gamma_{2,0}\Xi\|^2\nu_k\mathrm{d}y_k\mathrm{d}s\right).
\end{align}Next, we estimate each term on the right-hand side of \eqref{ppzz1}. For the third term on the right-hand side of \eqref{ppzz1}, by \eqref{b6}-\eqref{b7} and the Burkholder-Davis-Gundy inequality we have
\begin{align}
&2(\epsilon_1-\epsilon_{1,0})\mathbb{E}\left(\sup_{r\in[\tau,t]}\left|\int_\tau^r\sum_{k\in\mathbb{N}}\left\langle \aleph_k(s,u^\lambda(s)),u^\lambda(s)-u^{\lambda_0}(s)\right\rangle
\mathrm{d}W_k(s)\right|\right)\notag\\
\leq&2(\epsilon_1-\epsilon_{1,0})C\mathbb{E}\left(\left(\int_\tau^t\sum_{k\in\mathbb{N}}
\|\aleph_k(s,u^\lambda(s))\|^2\|u^\lambda(s)-u^{\lambda_0}(s)\|^2
\mathrm{d}s\right)^{\frac{1}{2}}\right)\notag\\
&2(\epsilon_1-\epsilon_{1,0})C\mathbb{E}\left(\sup_{r\in[\tau,t]}\|u^\lambda(s)-u^{\lambda_0}(s)\|
\left(\int_\tau^t \sum\|\aleph_k(s,u^\lambda(s))\|^2\mathrm{d}s\right)^{\frac{1}{2}}\right)\notag\\
\leq&\frac{1}{4}\mathbb{E}\left(\sup_{r\in[\tau,t]}\|u^\lambda(s)-u^{\lambda_0}(s)\|^2\right)
+4(\epsilon_1-\epsilon_{1,0})^2\int_\tau^t\sum_{k\in\mathbb{N}}\|\aleph_k(s,u^\lambda(s))\|^2\mathrm{d}s\notag\\
\leq&\frac{1}{4}\mathbb{E}\left(\sup_{r\in[\tau,t]}\|u^\lambda(s)-u^{\lambda_0}(s)\|^2\right)
+8(\epsilon-\epsilon_{1,0})^2\int_\tau^t\sum_{k\in\mathbb{N}}(\|h_k(s)\|^2
+\|\sigma_k(t,u^\lambda(s))\|^2)\mathrm{d}s\notag\\
\leq&\frac{1}{4}\mathbb{E}\left(\sup_{r\in[\tau,t]}\|u^\lambda(s)-u^{\lambda_0}(s)\|^2\right)
+8(\epsilon_1-\epsilon_{1,0})^2\int_\tau^t\sum_{k\in\mathbb{N}}(\|h_k(s)\|^2
)\mathrm{d}s\notag\\
&+16(\epsilon-\epsilon_{1,0})^2\alpha^2\|\delta\|^2\int_\tau^t(1+\|u^\lambda\|^2)\mathrm{d}s,
\end{align}and
\begin{align}
&\mathbb{E}\Big(\sup_{r\in[\tau,t]}\Big|\int_\tau^r\sum_{k\in\mathbb{N}}\int_{0<y_k<1}
\Big(\|u^\lambda(s-)-u^{\lambda_0}(s-)+(\epsilon_2-\epsilon_{2,0})\hat{\aleph}_k(s,u)
\notag\\
&+\epsilon_{2,0}\Xi\|^2-\|u^\lambda(s-)
-u^{\lambda_0}(s-)\|^2\Big)\tilde{L}_k(\mathrm{d}s,\mathrm{d}y_k)\Big|\Big)\notag\\
\leq&\mathbb{E}\left(\sup_{r\in[\tau,t]}\left|\int_\tau^r\sum_{k\in\mathbb{N}}\int_{0<y_k<1}
2\left\langle (\epsilon_2-\epsilon_{2,0})\hat{\aleph}_k(s,u)+\epsilon_{2,0}\Xi,u^\lambda(s-)
-u^{\lambda_0}(s-)\right\rangle \tilde{L}_k(\mathrm{d}s,\mathrm{d}y_k)\right|\right)\notag\\
&+\mathbb{E}\left(\sup_{r\in[\tau,t]}\left|\int_\tau^r\sum_{k\in\mathbb{N}}\int_{0<y_k<1}
\|(\epsilon_2-\epsilon_{2,0})\hat{\aleph}_k(s,u)+\epsilon_{2,0}\Xi\|^2\tilde{L}_k(\mathrm{d}s,\mathrm{d}y_k)
\right|\right)\notag\\
\leq&C\mathbb{E}\left(\left(\int_\tau^t\sum_{k\in\mathbb{N}}\int_{0<y_k<1}
\|(\epsilon_2-\epsilon_{2,0})\hat{\aleph}_k(s,u)+\epsilon_{2,0}\Xi\|^2
\|u^\lambda(s)
-u^{\lambda_0}(s)\|^2\nu_k\mathrm{d}y_k\mathrm{d}s
\right)^{\frac{1}{2}}\right)\notag\\
&+C^2\mathbb{E}\left(\left(\int_\tau^t\sum_{k\in\mathbb{N}}\int_{0<y_k<1}
\|(\epsilon_2-\epsilon_{2,0})\hat{\aleph}_k(s,u)+\epsilon_{2,0}\Xi\|^2
\nu_k\mathrm{d}y_k\mathrm{d}s
\right)\right)\notag\\
\leq&\frac{1}{4}\mathbb{E}\left(\sup_{r\in[\tau,t]}\|u^\lambda-u^{\lambda_0}\|^2\right)
+2C^2\mathbb{E}\left(\left(\int_\tau^t\sum_{k\in\mathbb{N}}\int_{0<y_k<1}
\|(\epsilon_2-\epsilon_{2,0})\hat{\aleph}_k(s,u)+\epsilon_{2,0}\Xi\|^2
\nu_k\mathrm{d}y_k\mathrm{d}s
\right)\right)\notag\\
\leq&4C^2\mathbb{E}\Big(\Big(\int_\tau^t\int_{0<y_k<1}\Big((\epsilon_2-\epsilon_{2,0})^2\sum_{k\in\mathbb{N}}\|
\kappa_k(s)+q(t,u^\lambda(s-),y_k)\|^2
\notag\\
&+\epsilon_{2,0}^2\sum_{k\in\mathbb{N}}\|q(s,u^\lambda(s),y_k)-q(s,u^{\lambda_0}(s),y_k)\|^2
\Big)\nu_k\mathrm{d}y_k\mathrm{d}s\Big)\Big)\notag\\
&+\frac{1}{4}\mathbb{E}\left(\sup_{r\in[\tau,t]}\|u^\lambda-u^{\lambda_0}\|^2\right)\notag\\
\leq&8C^2(\epsilon_2-\epsilon_{2,0})^2\mathbb{E}\Big(\int_\tau^t
\|\kappa_k(s)\|^2\mathrm{d}s\Big)+16C^2(\epsilon_2-\epsilon_{2,0})^2\alpha^2\|\delta\|^2\mathbb{E}\Big(\Big(\int_\tau^t
\Big(1+\|u^\lambda\|^2\Big)\mathrm{d}s\Big)
\notag\\
&+C\epsilon_{2,0}^2\mathbb{E}\Big(\Big(\int_\tau^t\|u^\lambda-u^{\lambda_0}\|^2
\Big)\nu_k\mathrm{d}y_k\mathrm{d}s\Big)\Big)\notag\\
&+\frac{1}{4}\mathbb{E}\left(\sup_{r\in[\tau,t]}\|u^\lambda-u^{\lambda_0}\|^2\right).
\end{align}For the fourth term on the right-hand side of the equation, by \eqref{b6} and the BDG inequality we obtian
\begin{align}
&2\epsilon_{1,0}\mathbb{E}\left(\sup_{r\in[\tau,t]}\left|\int_\tau^r\sum_{k\in\mathbb{N}} \left\langle\sigma_k(s,u^\lambda)-\sigma_k(s,u^{\lambda_0}), u^\lambda(s)-u^{\lambda_0}(s)\right\rangle\mathrm{d}W_k(s)\right|\right)\notag\\
\leq&\frac{1}{4}\mathbb{E}\left(\sup_{r\in[\tau,t]}\|u^\lambda(s)-u^{\lambda_0}(s)\|^2\right)
+C\epsilon_{1,0}^2\int_\tau^t\mathbb{E}\left(\sup_{r\in[\tau,s]}
\|u^\lambda(r)-u^{\lambda_0}(r)\|^2\right)\mathrm{d}s.
\end{align}For the fifth term on the right side of \eqref{ppzz1} we have
\begin{align}
&\mathbb{E}\left(\int_\tau^t\sum_{k\in\mathbb{N}}\|(\epsilon_1-\epsilon_{1,0})\aleph_k(s,u^\lambda(s)) +\epsilon_{1,0}(\sigma(s,u^\lambda)-\sigma(s,u^{\lambda_0}))\|^2\mathrm{d}s\right)\notag\\
\leq&C(\epsilon_1-\epsilon_{1,0})^2\int_\tau^t\mathbb{E}\left(2\alpha^2 \|\delta\|^2 (1 + \|u^\lambda(s)\|^2)\right)\mathrm{d}s\notag\\
&+C\epsilon_{1,0}^2\int_\tau^t\mathbb{E}\left(\sup_{r\in[\tau,s]}\|u^\lambda(s)-u^{\lambda_0}(s)\|^2\right)\mathrm{d}s
\end{align}To address the other terms ($v$-equation)  of \eqref{ppzz1}, the same approach can be applied. For all $t\in[\tau,\tau+T]$ and $\mathbb{E}(\|\xi^\lambda\|^2_{L^2(\Omega,\ell^2\times\ell^2)})\leq K(\tau)$, we have
\begin{align}\label{vvbb}
&\mathbb{E}\left(\sup_{r\in[\tau,t]}\|\varphi^\lambda-\varphi^{\lambda_0}\|^2\right)
\leq C(\lambda-\lambda_0)^2+C\int_\tau^t\mathbb{E}
\left(\sup_{r\in[\tau,t]}\|\varphi^\lambda-\varphi^{\lambda_0}\|^2
\right)\mathrm{d}s.
\end{align}where $C$ is any positive constants which may change from line to line.
Applying Gronwall's inequality to \eqref{vvbb} we have
\begin{align}
\mathbb{E}\left(\sup_{r\in[\tau,t]}\|\varphi^\lambda-\varphi^{\lambda_0}\|^2\right)\leq  C(\lambda-\lambda_0)^2e^{C(t-\tau)}.
\end{align}Then we obtain for $\tau \in \mathbb{R}$, $K(\tau)$ and $t \in \mathbb{R}^+$,
\begin{align}\label{h20}
\sup_{\mathbb{E}(\|\xi\|^2) \leq K^2(\tau)} \mathbb{E} \left( \|\varphi^{\lambda}(\tau + t, \tau, \xi) - \varphi^{\lambda_0}(\tau + t, \tau, \xi)\|^2 \right) \leq \vartheta(\lambda),
\end{align}where $\vartheta(\lambda) \to \lambda_0$ as $\lambda \to \lambda_0$.   Following the approach proposed in \cite{Lime1, Lime2}, given $\varepsilon > 0$, by the tightness of $\mathcal{A}(\tau)$, there is a compact set $C_\varepsilon$ in $X$ such that
\begin{align}\label{iiyyygf1}
\sup_{\mu \in \mathcal{A}(\tau)} \mu(X\setminus C_\varepsilon)
= \sup_{N \in \mathbb{N}} \sup_{\mu \in \mathcal{A}(\tau)} \widetilde{\mu}(X \setminus C_\varepsilon)
\leq \varepsilon.
\end{align}Noting that $|\phi(x)-\phi(\hat{x})|\le 2$ for all $x,\hat{x}\in X$ and $\phi\in L_b(X)$ with $\|\phi\|_{L_b}\le1$, we see from \eqref{iiyyygf1} that for all $N\in\mathbb{N}$,
\begin{align}
J^1 :=&
\sup_{\mu\in\mathcal{A}(\tau)}\sup_{\|\phi\|_{L_b}\le1}
\int_{X\setminus C_\varepsilon}
\mathbb{E}\!\left[\phi(\varphi^\lambda(\tau+T,\tau,y))
-\phi(\varphi^{\lambda_0}(\tau+T,\tau,y))\right]\mu_N(dz)\notag\\
\le& 2P(\Omega)\sup_{\mu\in\mathcal{A}(\tau)}\mu(X\setminus C_\varepsilon)
\le 2\varepsilon.
\end{align}Note that $|\phi(x)-\phi(\hat{x})|\le\|x-\hat{x}\|_X$ for $\|\phi\|_{L_b}\le1$ and $C_\varepsilon$ is compact, by \eqref{h20},  we have
\begin{align}
\begin{aligned}
J^2
&:=
\sup_{\mu\in\mathcal{A}(\tau)}\sup_{\|\phi\|_{L_b}\le1}
\int_{X\cap C_\varepsilon}
\mathbb{E}\left[\phi(\varphi^\lambda(\tau+T,\tau,y))
-\phi(\varphi^{\lambda_0}(\tau+T,\tau,y))\right]\mu(dy)\\
&\le
\sup_{\mu\in\mathcal{A}(\tau)}
\int_{X\cap C_\varepsilon}
\mathbb{E}\|\varphi^\lambda(\tau+T,\tau,y)
- \varphi^{\lambda_0}(\tau+T,\tau,y)\|_X\,\mu(dy)\\
&\le
\sup_{\mu\in\mathcal{A}(\tau)}
\int_{X\cap C_\varepsilon}
\bigl(\mathbb{E}\|\varphi^\lambda(\tau+T,\tau,y)
- \varphi^{\lambda_0}(\tau+T,\tau,y)\|_X^2\bigr)^{1/2}\mu(dy)\\
&<\varepsilon.
\end{aligned}
\end{align}Using $\mathrm{supp}\,\tilde{v}\subset X$ for $v\in\mathcal{M}(X)$, when $\lambda\rightarrow\lambda_0$, we have from
\begin{equation*}
\int_{X} \phi(x)[S(t,\tau)\mu](dx)
= \int_{X} \mathbb{E}\phi(\varphi(t,\tau,x))\mu(dx),
\quad \forall\, \mu \in \mathcal{M}(X),\, \phi \in C_b(X).
\end{equation*}that
\begin{align}
\begin{aligned}
&\sup_{\mu \in \mathcal{A}(\tau)}
d_{\mathcal{M}(X)} \left( \widetilde{S}^\lambda(\tau+T,\tau)\mu,\, S^{\lambda_0}(\tau+T,\tau)\tilde{\mu} \right)\notag\\
=& \sup_{\mu \in \mathcal{A}(\tau)}
\sup_{\|\phi\|_{L_b} \leq 1}
\left| \int_X \phi(x) \left[\widetilde{S}(\tau+T,\tau)\mu\right](dx)
- \int_X \phi(x) \left[S(\tau+T,\tau)\tilde{\mu}\right](dx) \right| \\
= &\sup_{\mu \in \mathcal{A}(\tau)}
\sup_{\|\phi\|_{L_b} \leq 1}
\left| \int_{X} \phi(y)\left[S(\tau+T,\tau)\mu\right](dy)
- \int_X \mathbb{E} \phi(\varphi^{\lambda_0}(\tau+T,\tau,x)) \tilde{\mu}(dx) \right| \\
=& \sup_{\mu \in \mathcal{A}(\tau)}
\sup_{\|\phi\|_{L_b} \leq 1}
\left| \int_{X} \mathbb{E} \phi(\varphi^\lambda(\tau+T,\tau,y)) \mu(dy)
- \int_{X} \mathbb{E} \phi(\varphi^{\lambda_0}(\tau+T,\tau,y)) \mu(dy) \right| \\
\leq& J^1 + J^2 < 3\varepsilon,
\end{aligned}
\end{align}The whole proof is complete.
\end{proof}
By Lemma \ref{h17}, one can verify that
\begin{align}\label{h21}
K = \left\{ K(\tau) = \bigcup_{\lambda \in [0, 1]\times[0, 1]\times[0, 1]\times[0, 1]} \mathcal{A}_{\lambda}(\tau) : \tau \in \mathbb{R} \right\} \in \mathcal{D}.
\end{align}Then the main result of this section are given below.
\begin{theorem}
Suppose \eqref{b6}-\eqref{b7}, \eqref{bc0}, and \eqref{xs} hold. Then for $\tau \in \mathbb{R}$,
\begin{align}\label{h23}
\lim_{\lambda \to \lambda_0} d_{\mathcal{P}_2(\ell^2\times\ell^2)}(\mathcal{A}_{\lambda}(\tau), \mathcal{A}_{\lambda_0}(\tau)) = 0.
\end{align}
\end{theorem}

\begin{proof}
Set $K^\lambda = \{K^\lambda(\tau): \tau \in \mathbb{R}\}$, defined by \eqref{h21}, serves as a closed $\mathcal{D}$-pullback absorbing set for $S^{\lambda}$ uniformly for all $\lambda \in [0,1]^4$. Specifically, for every $\tau \in \mathbb{R}$ and $\mathcal{D} = \{\mathcal{D}(\tau): \tau \in \mathbb{R}\} \in \mathcal{D}$, there exists a time $T_1 = T_1(\tau, \mathcal{D}) > 0$ satisfying
\begin{align}
\label{newindependenty}
\bigcup_{t \ge T_1} \bigcup_{\lambda \in [0,1]^4} S^{\lambda}(t, \tau - t)\, \mathcal{D}(\tau - t) \subseteq K(\tau).
\end{align}

This, combined with the invariance of $\mathcal{A}^{\lambda}$, directly yields that $\mathcal{A}^{\lambda}(\tau) \subseteq K^\lambda(\tau)$ for all $\lambda \in [0,1]^4$ and $\tau \in \mathbb{R}$.

Next, let $\eta > 0$ be arbitrary. Since $\mathcal{K^\lambda} = \{\mathcal{K^\lambda}(\tau): \tau \in \mathbb{R}\} \in \mathcal{D}$ and $\mathcal{A}^\lambda = \{\mathcal{A}^\lambda(\tau): \tau \in \mathbb{R}\} \in \mathcal{D}$ is  \textbf{$\mathcal{D}$-PMAs} for $S$ in $(\mathcal{P}_2(\ell^2\times\ell^2), d_{\mathcal{P}(\ell^2\times\ell^2)})$, it follows that there exists a time $T_2 = T_2(\tau, K, \mathcal{A}) > 0$ such that
\begin{align}
\label{newindependenty1}
\sup_{\mu \in \mathcal{K}(\tau - T_2)} d_{\mathcal{P}(\ell^2\times\ell^2)}\bigl(S(T_2, \tau - T_2)\mu,\, \mathcal{A}(\tau)\bigr)
=
d_{\mathcal{P}(\ell^2\times\ell^2)}\bigl(S(T_2, \tau - T_2)\, K^\lambda(\tau - T_2),\, \mathcal{A}(\tau)\bigr)
< \frac{\eta}{2}.
\end{align}

By \eqref{h22io},
\begin{align}
\label{newindependenty2}
\sup_{\mu \in \mathcal{K}(\tau - T_2)}
d_{\mathcal{P}(\ell^2\times\ell^2)}\bigl(S^{\chi}(T_2, \tau - T_2)\mu,\; S(T_2, \tau - T_2)\mu\bigr)
< \frac{\eta}{2}.
\end{align}

Hence, using \eqref{newindependenty1}, \eqref{newindependenty2}, and the fact that $\mathcal{A}^{\lambda}(\tau - T_2) \subseteq K^\lambda(\tau - T_2)$ for all $\lambda \in [0,1]^4$, we obtain that for all $\lambda \in (0, 1)$,
\begin{align*}
&\sup_{\mu \in \mathcal{A}^{\lambda}(\tau - T_2)}  d_{\mathcal{P}(\ell^2\times\ell^2)}\bigl(S^{\lambda}(T_2, \tau - T_2)\mu,\, \mathcal{A}(\tau)\bigr)\notag \\
\le&
\sup_{\mu \in \mathcal{K}(\tau - T_2)} d_{\mathcal{P}(\ell^2\times\ell^2)}\bigl(S^{\lambda}(T_2, \tau - T_2)\mu,\, S(T_2, \tau - T_2)\mu\bigr) \\
& +
\sup_{\mu \in \mathcal{K}(\tau - T_2)} d_{\mathcal{P}(\ell^2\times\ell^2)}\bigl(S(T_2, \tau - T_2)\mu,\, \mathcal{A}(\tau)\bigr)
< \eta.
\end{align*}In view of the invariance of $\mathcal{A}^{\lambda}$, it follows that for all $\lambda \in (0, 1)$,
\begin{align*}
d_{\mathcal{P}(\ell^2\times\ell^2)}\bigl(\mathcal{A}^{\lambda}(\tau),\, \mathcal{A}(\tau)\bigr)
=
\sup_{\mu \in \mathcal{A}^{\lambda}(\tau)}
d_{\mathcal{P}(\ell^2\times\ell^2)}\bigl(\mu,\, \mathcal{A}(\tau)\bigr)
< \eta.
\end{align*}This completes the proof.
\end{proof}

\section{Example for numerical simulations}
In this section, we give a   numerical simulation for a one-dimensional stochastic ODE:
\begin{align}\label{v0s}
du(t) &= \left[ -2.5\,u(t) + u^2(t)\,v(t) - u^3(t) + e^{-t} \right]\,dt
+ \epsilon_1 \left[ h(t) + \sigma(t, u(t)) \right]\,dW(t) \notag \\
&\quad + \epsilon_2 \int_{\mathbb{R}} \left[ \kappa(t) + q(t, u(t), y) \right]\,\tilde{L}(dy, dt), \notag\\
dv(t) &= \left[ -v(t) - u^2(t)\,v(t) + u^3(t) + e^{-t} \right]\,dt
+ \gamma_1 \left[ h(t) + \sigma(t, v(t)) \right]\,dW(t) \notag \\
&\quad + \gamma_2 \int_{\mathbb{R}} \left[ \kappa(t) + q(t, v(t), y) \right]\,\tilde{L}(dy, dt),
\end{align}with $u(0)= 2.0, v(0)=0.0$, where $p=1$, $\lambda \in [0, 1]^4$,
\begin{align*}
h(t) &= \cos(2t), \quad
\sigma(t, x) = e^{-t^2} x^2, \quad
\kappa(t) = \sin(2t), \quad
q(t, x, y) = \frac{e^{-t^2} x^2}{1 + y^2}.
\end{align*}The numerical scheme is based on the Euler--Maruyama method adapted for SDEs with L\'evy noise. The simulation was performed on the interval \([0, T]\) with \(T = 10.0\), using a fixed time step \(\Delta t = 0.001\), yielding \(N = 10{,}000\) time steps. The L\'evy noise was modeled as a \emph{compound Poisson process} with jump intensity $\lambda_{\text{Poisson}} = 2.0$ and Gaussian jump amplitude $y_{n,j} \sim \mathcal{N}(0, \sigma_{\text{jump}}^2)$ with $\sigma_{\text{jump}} = 1.0$. The Brownian increment was simulated as $\Delta W_n \sim \mathcal{N}(0, \Delta t)$.

The discretized scheme takes the form:
\[
\begin{aligned}
u_{n+1} &= u_n + \left[-2.5u_n + u_n^2v_n - u_n^3 + e^{-t_n} \right]\Delta t
+ \varepsilon_1\left[\cos(2t_n) + e^{-t_n^2}u_n^2\right]\Delta W_n \\
&\quad + \varepsilon_2 \sum_{j=1}^{P_n} \left[ \sin(2t_n) + \frac{e^{-t_n^2}u_n^2}{1 + y_{n,j}^2} \right] y_{n,j},\\
v_{n+1} &= v_n + \left[-v_n - u_n^2v_n + u_n^3 + e^{-t_n} \right]\Delta t
+ \varepsilon_1\left[\cos(2t_n) + e^{-t_n^2}v_n^2\right]\Delta W_n \\
&\quad + \varepsilon_2 \sum_{j=1}^{P_n} \left[ \sin(2t_n) + \frac{e^{-t_n^2}v_n^2}{1 + y_{n,j}^2} \right] y_{n,j},
\end{aligned}
\]
where \(P_n \sim \text{Poisson}(\lambda_{\text{Poisson}} \Delta t)\) is the number of jumps, and \(y_{n,j}\) are the jump magnitudes.

\vspace{-8pt}
\begin{center}
    \includegraphics[width=0.9\textwidth, height=5cm]{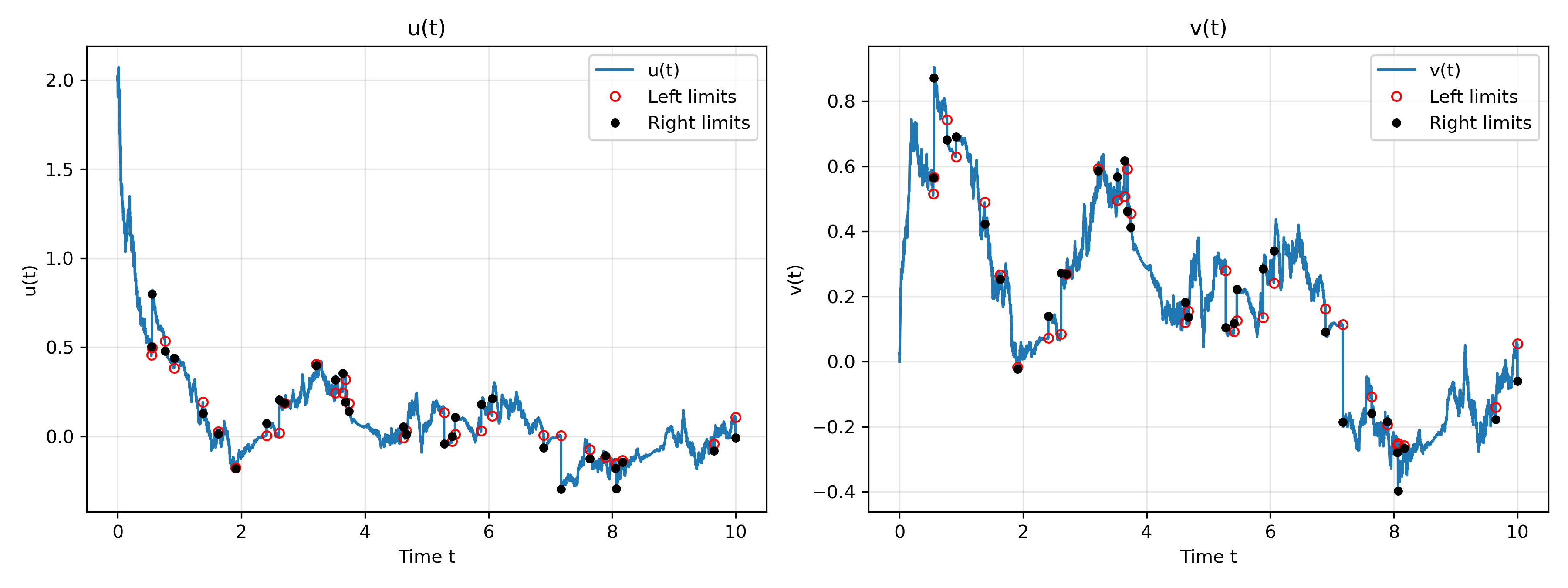}\\
    \small\textbf{Figure 4:} The sample path of L\'{e}vy processes ($\lambda\neq\lambda_0$)
    \label{fig:cadlag_path4}
\end{center}
\vspace{-8pt}
\vspace{-8pt}
\begin{center}
    \includegraphics[width=0.9\textwidth, height=5cm]{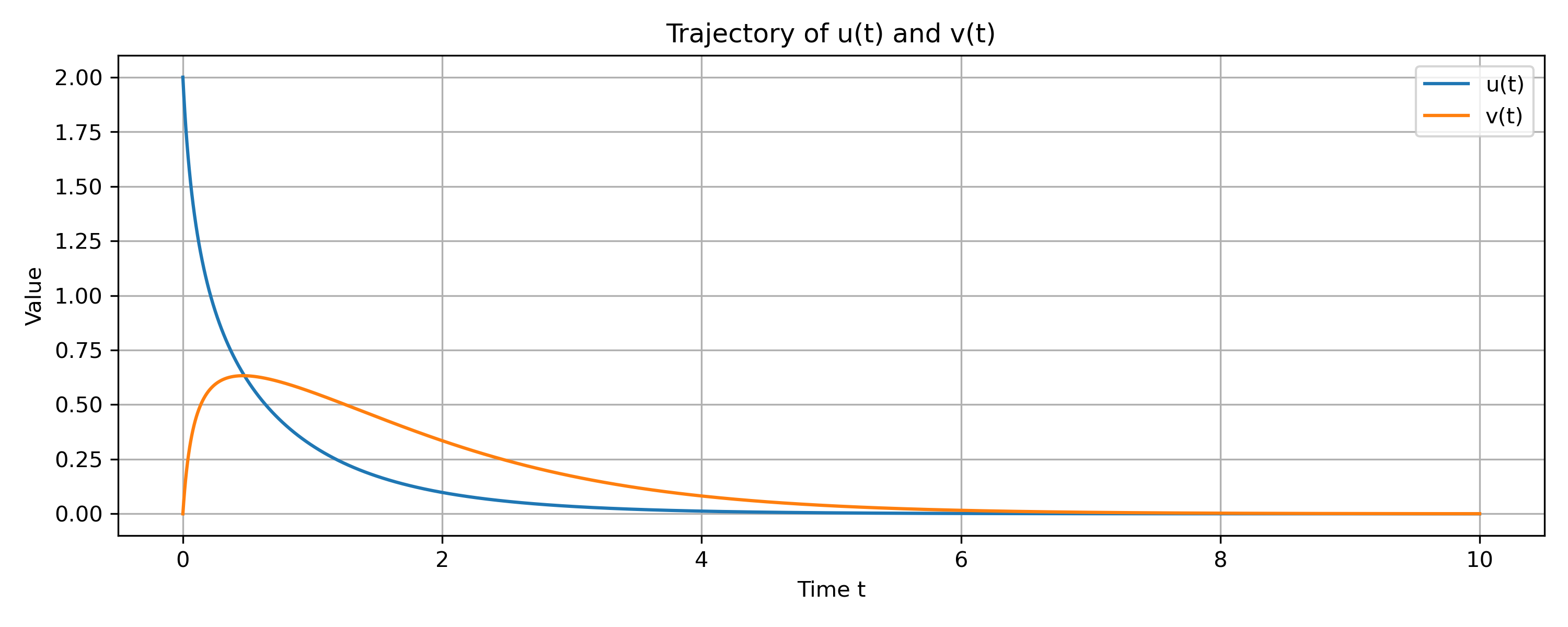}\\
    \small\textbf{Figure 5:} The sample path of L\'{e}vy processes ($\lambda=\lambda_0$)
    \label{fig:cadlag_path5}
\end{center}
\vspace{-8pt}
By combining the Figure 3, Figure 4, Figure 5, we present simulations for both the deterministic ODE system ($\lambda = \lambda_0$) and the stochastic ODE system ($\lambda \neq \lambda_0$) in the figure. As shown, the ODE trajectories fluctuate around the deterministic ones due to noise, remaining close overall, and the distance between the two trajectories decreases as $\epsilon$ tends to zero. This numerically illustrates that, in some sense, $\mathcal{A}^{\lambda}(\tau) \rightarrow \mathcal{A}^{\lambda_0}(\tau)$ as $\lambda \to \lambda_0$.

\section*{Acknowledgments and Funding}
This work is supported by  National Natural Science Foundation of China (12161019, 12361046), Guizhou Provincial Basic Research Program (Natural Science) (QKHJC-ZK [2022] YB 318), Natural Science Research Project of Guizhou Provincial Department of Education (QJJ [2023] 011).
\section*{Data availability}

No data was used for the research described in the article.
\section*{Author Contributions}

\noindent
\textbf{Guofu Li}: Writing-review \& editing, Writing-original draft;
\textbf{Jianxin Wu}: Writing-review \& editing;
\textbf{Yunshun Wu}: Writing-review \& editing.

\section*{Declarations}

\textbf{Conflict of interest}\quad No conflict of interest exists in the submission of this manuscript.

\end{document}